\documentclass[11pt]{article}
\usepackage{amsmath,amsthm,amsfonts,amssymb,enumerate,calc,graphicx,color}
\usepackage[colorlinks=true,citecolor=black,linkcolor=black,urlcolor=blue]{hyperref}
\usepackage[lmargin=30mm,rmargin=30mm,bmargin=30mm,tmargin=30mm]{geometry}
\usepackage[numbers,sort&compress]{natbib}
\renewcommand{\baselinestretch}{1.1}
\usepackage[hang]{footmisc} 

\renewcommand{\thefootnote}{\fnsymbol{footnote}}	
\allowdisplaybreaks
\newcommand\DateFootnote{
\begingroup
\renewcommand\thefootnote{}
\footnote{June 7, 2013. Revised: \today}
\setcounter{footnote}{0}
\vspace*{-3ex}
\endgroup}

\makeatletter
\renewcommand\section{\@startsection {section}{1}{\z@}%
                                   {-3ex \@plus -1ex \@minus -.2ex}%
                                   {2ex \@plus.2ex}%
                                   {\normalfont\large\bfseries}}
\renewcommand\subsection{\@startsection{subsection}{2}{\z@}%
                                     {-2.5ex\@plus -1ex \@minus -.2ex}%
                                     {1.5ex \@plus .2ex}%
                                     {\normalfont\normalsize\bfseries}}
\renewcommand\subsubsection{\@startsection{subsubsection}{3}{\z@}%
                                     {-2ex\@plus -1ex \@minus -.2ex}%
                                     {1ex \@plus .2ex}%
                                     {\normalfont\normalsize\bfseries}}
 \renewcommand\paragraph{\@startsection{paragraph}{4}{\z@}%
                                    {1.5ex \@plus.5ex \@minus.2ex}%
                                    {-1em}%
                                    {\normalfont\normalsize\bfseries}}
\renewcommand\subparagraph{\@startsection{subparagraph}{5}{\parindent}%
                                       {1.5ex \@plus.5ex \@minus .2ex}%
                                       {-1em}%
                                      {\normalfont\normalsize\bfseries}}
\makeatother

\newcommand{\arXiv}[1]{arXiv:\,\href{http://arxiv.org/abs/#1}{#1}}
\newcommand{\msn}[1]{MR:\,\href{http://www.ams.org/mathscinet-getitem?mr=MR#1}{#1}}
\newcommand{\MSN}[2]{MR:\,\href{http://www.ams.org/mathscinet-getitem?mr=MR#1}{#1}}
\newcommand{\doi}[1]{doi:\,\href{http://dx.doi.org/#1}{#1}}
\newcommand{\Zbl}[1]{Zbl:\,\href{http://www.zentralblatt-math.org/zmath/en/search/?q=an:#1}{#1}}

\theoremstyle{plain}
\newtheorem{thm}{Theorem}
\newtheorem{lem}[thm]{Lemma}

\newcommand{\preproof}{\vspace*{-3ex}}
\newcommand{\seclabel}[1]{\label{sec:#1}}

\newcommand{\secref}[1]{\mbox{Section~\ref{sec:#1}}}
\newcommand{\lemlabel}[1]{\label{lem:#1}}
\newcommand{\lemref}[1]{Lemma~\ref{lem:#1}}
\newcommand{\twolemref}[2]{Lemmas~\ref{lem:#1} and~\ref{lem:#2}}
\newcommand{\thmlabel}[1]{\label{thm:#1}}
\newcommand{\thmref}[1]{Theorem~\ref{thm:#1}}
\newcommand{\twothmref}[2]{Theorems~\ref{thm:#1} and~\ref{thm:#2}}

\newcommand{\figlabel}[1]{\label{fig:#1}}

\newcommand{\figref}[1]{\mbox{Figure~\ref{fig:#1}}}

\newcommand{\applabel}[1]{\label{app:#1}}
\newcommand{\appref}[1]{Appendix~\ref{app:#1}}

\newcommand{\Oh}[1]{\ensuremath{\protect\mathcal{O}(#1)}}
\newcommand{\etal}{et al.}
\DeclareMathOperator{\depth}{depth}
\DeclareMathOperator{\dist}{dist}
\DeclareMathOperator{\blah}{col}
\DeclareMathOperator{\lab}{label}
\DeclareMathOperator{\tn}{tn}
\DeclareMathOperator{\qn}{qn}

\newcommand{\N}{\mathbb{N}}

\newcommand{\CEIL}[1]{\ensuremath{\protect\left\lceil#1\right\rceil}}
\newcommand{\ceil}[1]{\lceil{#1}\rceil}
\newcommand{\floor}[1]{\lfloor{#1}\rfloor}
\newcommand{\half}{\ensuremath{\protect\tfrac{1}{2}}}
\renewcommand{\geq}{\geqslant}
\renewcommand{\leq}{\leqslant}

\begin{document}

\vspace*{2ex}
{\Large\bfseries\boldmath\scshape Layered Separators in Minor-Closed Graph Classes with Applications}

\DateFootnote

\medskip
\bigskip
{\large 
Vida Dujmovi{\'c}\,\footnotemark[2]
\quad
Pat Morin\footnotemark[3] 
\quad 
David~R.~Wood\,\footnotemark[4]
}

\bigskip
\bigskip
\emph{Abstract.} Graph separators are a ubiquitous tool in graph theory and computer science. However, in some applications, their usefulness is limited by the fact that the separator can be as large as $\Omega(\sqrt{n})$ in graphs with $n$ vertices. This is the case for planar graphs, and more generally, for proper minor-closed classes. We study a special type of graph separator, called a \emph{layered separator}, which may have linear size in $n$, but has bounded size with respect to a different measure, called the \emph{width}. We prove, for example,  that planar graphs and graphs of bounded Euler genus admit layered separators of bounded width. More generally, we characterise the minor-closed classes that admit layered separators of bounded width as those that exclude a fixed apex graph as a minor. 

\medskip
We use layered separators to prove \Oh{\log n} bounds for a number of problems where \Oh{\sqrt{n}} was a long-standing previous best bound. This includes the \emph{nonrepetitive chromatic number} and \emph{queue-number} of graphs with bounded Euler genus. We extend these results with a \Oh{\log n} bound on the nonrepetitive chromatic number of graphs excluding a fixed topological minor, and a $\log^{\Oh{1}}n$ bound on the queue-number of graphs excluding a fixed minor. Only for planar graphs were $\log^{\Oh{1}}n$ bounds previously known. Our results imply that every $n$-vertex graph excluding a fixed minor has a \emph{3-dimensional grid drawing} with $n\log^{\Oh{1}}n$ volume, whereas the previous best bound was \Oh{n^{3/2}}. 


\footnotetext[2]{School of Computer Science and Electrical Engineering,  
University of Ottawa, Ottawa, Canada (\texttt{vida.dujmovic@uottawa.ca}). Research  supported by NSERC and the Ontario Ministry of Research and Innovation.}

\footnotetext[3]{School of  Computer Science, Carleton University,  Ottawa, Canada (\texttt{morin@scs.carleton.ca}). Research  supported by NSERC. }

\footnotetext[4]{School of Mathematical Sciences, Monash   University, Melbourne, Australia  (\texttt{david.wood@monash.edu}). Research supported by the Australian Research Council.}

\footnotetext{A short version of this paper and reference \Citep{Duj15} was presented at the 54th Annual IEEE Symposium on Foundations of Computer Science (FOCS '13).}



\newpage
\tableofcontents
\newpage

\renewcommand{\thefootnote}{\arabic{footnote}}
\setlength{\parskip}{2ex}

\section{Introduction}

Graph separators are a ubiquitous tool in graph theory and computer science since they are key to many divide-and-conquer and dynamic programming algorithms. Typically, the smaller the separator the better the results  obtained. For instance, many problems that are $\mathcal{NP}$-complete for general graphs have polynomial time solutions for classes of graphs that have bounded size  separators---that is, graphs of bounded treewidth. 

By the classical result of Lipton and Tarjan \cite{LT79}, every $n$-vertex planar graph has a separator of size $\Oh{\sqrt{n}}$.  More generally, the same is true for every proper minor-closed graph class\footnote{A graph $H$ is a \emph{topological minor} of a graph $G$ if a subdivision of $H$ is a subgraph of $G$.  A graph $H$ is a \emph{minor} of a graph $G$ if a graph isomorphic to $H$ can be obtained from a subgraph of $G$ by contracting edges. A class $\mathcal{G}$ of graphs is \emph{minor-closed} if $H\in\mathcal{G}$ for every minor $H$ of $G$ for every graph $G\in\mathcal{G}$. A minor-closed class is \emph{proper} if it is not the class of all graphs.}, as proved by \citet{AST90}. While these results have found widespread use, separators of size $\Theta(\sqrt{n})$, or non-constant separators in general,  are not small enough to be useful in some applications. 

In this paper we study a type of graph separator, called layered separators, that may have $\Omega(n)$ vertices but have bounded size with respect to a different measure. In particular, layered separators intersect each  layer of some predefined vertex layering in a bounded number of vertices. We prove that many classes of graphs admit such separators, and we show how  (with simple proofs) they can be used to obtain logarithmic bounds for a variety of applications for which \Oh{\sqrt{n}} was the best known long-standing bound. These applications include nonrepetitive graph colourings, track layouts, queue layouts and 3-dimensional grid drawings of graphs. 

In the remainder of the introduction, we define layered separators, and describe our results on the classes of graphs that admit them. Following that, we describe the implications that these results have on the above-mentioned applications. 

\subsection{Layered Separations}
\seclabel{tool}

A \emph{layering} of a graph $G$ is a partition $(V_0,V_1,\dots,V_t)$ of $V(G)$ such that for every edge $vw\in E(G)$, if $v\in V_i$ and $w\in V_j$, then $|i-j|\leq 1$. Each set $V_i$ is called a \emph{layer}.  For example, for a vertex $r$ of a connected graph $G$, if $V_i$ is the set of vertices at distance $i$ from $r$, then $(V_0,V_1,\dots)$ is a layering of $G$, called the \emph{bfs layering} of $G$ starting from $r$. A \emph{bfs tree} of $G$ rooted at $r$ is a spanning tree  of $G$ such that for every vertex $v$ of $G$, the distance between $v$ and $r$ in $G$ equals the distance between $v$ and $r$ in $T$. Thus, if $v\in V_i$ then the $vr$-path in $T$ contains exactly one vertex from layer $V_j$ for $j\in\{0,\dots,i\}$.

A \emph{separation} of a graph $G$ is a pair $(G_1,G_2)$ of subgraphs of $G$ such that $G=G_1\cup G_2$.  In particular,   there is no edge  between $V(G_1) \setminus V(G_2)$ and $V(G_2) \setminus V(G_1)$.  The \emph{order} of a separation $(G_1,G_2)$ is  $|V(G_1\cap G_2)|$. 

A graph $G$  \emph{admits layered separations of width} $\ell$ with respect to a layering $(V_0,V_1,\dots,V_t)$ of $G$ if for every set $S\subseteq V(G)$, there is a separation $(G_1,G_2)$  of $G$ such that:
  \begin{itemize}
  \item for $i\in\{0,1,\dots,t\}$, layer $V_i$ contains at most $\ell$ vertices in $V(G_1 \cap G_2)$, and
  \item both $V(G_1) \setminus V(G_2)$ and $V(G_2) \setminus V(G_1)$ contain at most $\frac{2}{3}|S|$ vertices in $S$.
  \end{itemize}
Here the set  $V(G_1 \cap G_2)$ is called a \emph{layered separator of width $\ell$} of $G[S]$. Note that these separators do not necessarily have small order, in particular $V(G_1 \cap G_2)$ can have $\Omega(n)$ vertices. For brevity, we say a graph $G$  \emph{admits layered separations of width} $\ell$ if $G$ admits layered separations of width $\ell$  with respect to some layering of $G$. 

Layered separations are implicit in the seminal work of \citet{LT79} on separators in planar graphs, and in many subsequent papers (such as \citep{GHT-JAlg84,AS-SJDM96}). This definition was first made explicit by \citet{DFJW13}, who showed that a result of \citet{LT79} implies that every planar graph admits layered separations of width $2$.  This result was used by Lipton and Tarjan as a subroutine in their \Oh{\sqrt{n}} separator result. We generalise this result for planar graphs to graphs embedded on arbitrary surfaces.\!\footnote{The \emph{Euler genus} of a surface $\Sigma$ is $2 - \chi$, where $\chi$ is  the Euler characteristic of $\Sigma$. Thus the orientable surface with $h$ handles has Euler genus $2h$, and the non-orientable surface with $c$ cross-caps has Euler genus $c$. The \emph{Euler genus} of a graph $G$ is the minimum Euler genus of a surface in which $G$ embeds. See \citep{MoharThom} for background on graphs embedded in surfaces.} In particular, we prove that  graphs of Euler genus $g$ admit layered separations of width \Oh{g} (\thmref{GenusLayering} in \secref{Surfaces}). A key to this proof is the notion of a layered tree decomposition, which is of independent interest, and is introduced in \secref{DecompSep}. 

We further generalise this result by exploiting Robertson and Seymour's graph minor structure theorem. Roughly speaking, a graph $G$ is almost-embeddable in a surface $\Sigma$ if by deleting a bounded number of `apex' vertices, the remaining graph can be embedded in $\Sigma$, except for a bounded number of `vortices', where crossings are allowed in a well-structured way; see \secref{Vortices} where all these terms are defined. Robertson and Seymour proved that every graph from a proper minor-closed class can be obtained from clique-sums of graphs that are almost-embeddable in a surface of bounded Euler genus. Here, apex vertices can be adjacent to any vertex in the graph. However, such freedom is not possible for graphs that admit layered separations of bounded width. For example, the planar $\sqrt{n}\times\sqrt{n}$ grid plus one dominant vertex (adjacent to every other vertex) does not admit layered separations of width $o(\sqrt{n})$; see \secref{Vortices}. We define the notion of strongly almost-embeddable graphs, in which apex vertices are only allowed to be adjacent to vortices and other apex vertices. With this restriction, we prove that graphs obtained from clique-sums of strongly almost-embeddable graphs admit layered separations of bounded width (\thmref{kAlmost} in \secref{Vortices}). A recent structure theorem of \citet{DvoTho} says that $H$-minor-free graphs have this structure, for each apex\footnote{A graph $H$ is \emph{apex} if $H-v$ is planar for some vertex $v$.} graph $H$. We conclude that a minor-closed class $\mathcal{G}$ admits layered separations of bounded width if and only if $\mathcal{G}$ excludes some fixed apex graph. Then, in all the applications that we consider, we deal with (unrestricted) apex vertices separately, leading to \Oh{\log n} or $\log^{\Oh{1}}n$ bounds for every proper minor-closed class. These extensions depend on two tools of independent interest (rich tree decompositions and shadow-complete layerings) that are presented in \secref{GeneralMinor}. 

\subsection{Queue-Number and 3-Dimensional Grid Drawings}
\seclabel{IntroQueueDrawing}

Let $G$ be a graph.  In a linear ordering $\preceq$ of $V(G)$, two edges $vw$ and $xy$ are \emph{nested} if $v\prec x\prec y \prec w$. A \emph{$k$-queue layout} of a graph $G$ consists of a linear ordering $\preceq$ of $V(G)$ and a partition $E_1,\dots,E_k$ of $E(G)$, such that no two edges in each set $E_i$ are nested with respect to $\preceq$. The \emph{queue-number} of a graph $G$ is the minimum integer $k$ such that $G$ has a $k$-queue layout, and is denoted by $\qn(G)$. Queue layouts were introduced by Heath~\etal~\citep{HLR92,HR92} and have since been widely studied, with  applications in parallel process scheduling, fault-tolerant processing, matrix computations,  and sorting networks; see \cite{Pemmaraju-PhD,DujWoo-DMTCS04} for surveys. 


A number of classes of graphs are known to have bounded queue-number. For example, every tree has a 1-queue layout  \citep{HR92}, every outerplanar graph has a 2-queue layout \citep{HLR92}, every series-parallel graph has a 3-queue layout \citep{RM-COCOON95}, every graph with bandwidth $b$ has a $\ceil{\frac{b}{2}}$-queue layout \citep{HR92}, every graph with pathwidth $p$ has a $p$-queue layout \citep{DMW05}, and  more generally every graph with bounded treewidth has bounded queue-number \citep{DMW05}. All these classes have bounded treewidth. Only a few highly structured graph classes of unbounded treewidth, such as grids and cartesian products \citep{Wood-Queue-DMTCS05}, are known to have bounded queue-number. In particular, it is open whether planar graphs have bounded queue-number, as conjectured by Heath~\etal~\cite{HR92,HLR92}. 

The dual concept of a queue layout is a \emph{stack layout}, introduced by \citet{Ollmann73} and commonly called a \emph{book  embedding}. It is defined similarly, except that no two edges in the same set of the edge-partition are allowed to cross with respect to the vertex ordering (in contrast to queue layouts, which exclude nested edges in the same set). \emph{Stack-number} (also known as \emph{book thickness} or \emph{page-number}) is bounded for planar graphs \cite{Yannakakis89}, for graphs of bounded Euler genus \cite{Malitz94b}, and for every proper minor-closed class \cite{Blankenship-PhD03}. A recent construction of bounded degree monotone expanders by Bourgain and Yehudayoff~\citep{Bourgain09,BY13} has bounded stack-number and bounded  queue-number; see \citep{DW-ToC10,DMS14,DSW16}.  

Until recently, the best known upper bound for the queue-number of planar graphs was  \Oh{\sqrt{n}}. This upper bound follows easily from the fact that planar graphs have pathwidth at most \Oh{\sqrt{n}}. In a breakthrough result, this  bound was reduced to \Oh{\log^2 n} by \citet*{DFP13}, which was further improved by  \citet{Duj15} to \Oh{\log n} using a simple proof based on layered separators. In particular, \citet{Duj15} proved that every $n$-vertex graph that  admits layered separations of width $\ell$ has \Oh{\ell\log n} queue-number. Since every planar graph admits layered separations of width $2$, planar graphs have \Oh{\log n} queue-number~\citep{Duj15}. Moreover, we immediately obtain logarithmic bounds on the queue-number for the graph classes described in \secref{tool}. In particular, we prove that graphs with Euler genus $g$ have \Oh{g\log n} queue-number (\thmref{GenusTrack}), and graphs that exclude a fixed apex graph as a minor have \Oh{\log n} queue-number (\thmref{ApexTrack}). Furthermore, we extend this result  to all proper minor-closed classes with an upper bound of  $\log^{\Oh{1}}n$ (\thmref{GeneralTrack}).  The previously best known bound for all these classes, except for planar graphs,  was \Oh{\sqrt{n}}. 

One motivation for studying queue layouts is their connection with 3-dimensional graph drawing. A \emph{3-dimensional grid drawing} of a graph $G$ represents the vertices of $G$ by distinct grid points in $\mathbb{Z}^3$ and represents each edge of $G$ by the open segment between its endpoints so that no two edges intersect. The \emph{volume} of a 3-dimensional grid drawing is the number of grid points in the smallest axis-aligned grid-box that encloses the drawing. For example,  \citet{CELR-Algo96} proved that  the complete graph $K_n$ has a 3-dimensional grid drawing with volume \Oh{n^3} and this bound is optimal. \citet{PTT99} proved that every graph with bounded chromatic number has a 3-dimensional grid drawing with volume \Oh{n^2}, and this bound is optimal for $K_{n/2,n/2}$. More generally, \citet{BCMW-JGAA04} proved that every 3-dimensional grid drawing of an $n$-vertex $m$-edge graph has volume  at least $\frac{1}{8}(n+m)$. \citet{DujWoo-SubQuad-AMS} proved that every graph with bounded maximum degree has a 3-dimensional grid drawing with volume \Oh{n^{3/2}}, and the same bound holds for graphs from a proper minor-closed class. In fact, every graph with bounded degeneracy has a 3-dimensional grid drawing with \Oh{n^{3/2}} volume \citep{DujWoo-Order06}. \citet{DMW05} proved that every graph with bounded treewidth has a 3-dimensional grid drawing with volume \Oh{n}. Whether planar graphs have 3-dimensional grid drawings with \Oh{n} volume is a major open problem, due to \citet{FLW-GD01-ref}. The best known bound on the volume of 3-dimensional grid drawings of planar graphs is \Oh{n\log n} by \citet{Duj15}. We prove a \Oh{n\log n} volume bound for graphs of bounded Euler genus (\thmref{GenusDrawing}), and more generally, for apex-minor-free graphs (\thmref{ApexDrawing}). Most generally, we prove an  $n\log^{\Oh{1}}n$ volume bound for every proper minor-closed class (\thmref{MinorDrawing}). 

All our results about queue layouts are proved in \secref{TrackQueue}, and all  our results about 3-dimensional grid drawings are proved in \secref{3DimDrawing}. 

\subsection{Nonrepetitive Graph Colourings}
\seclabel{NonrepetitiveColourings}

A vertex colouring of a graph is \emph{nonrepetitive} if there is no path for which the first half of the path is assigned the same
sequence of colours as the second half.  More precisely, a $k$-\emph{colouring} of a graph $G$ is a function $\psi$ that assigns
one of $k$ colours to each vertex of $G$.  A path $(v_1,v_2,\dots,v_{2t})$ of even order in $G$ is \emph{repetitively} coloured by
$\psi$ if $\psi(v_i)=\psi(v_{t+i})$ for $i\in\{1,\dots,t\}$. A colouring $\psi$ of $G$ is \emph{nonrepetitive} if no path of $G$ is repetitively coloured by $\psi$. Observe that a nonrepetitive colouring is \emph{proper}, in the sense that adjacent vertices are coloured differently. The
\emph{nonrepetitive chromatic number} $\pi(G)$ is the minimum integer $k$ such that $G$ admits a nonrepetitive $k$-colouring.

The seminal result in this area is by \citet{Thue06}, who proved in 1906 that every path is nonrepetitively 3-colourable. Nonrepetitive colourings have recently been widely studied; see the surveys \citep{Gryczuk-IJMMS07,Grytczuk-DM08,CSZ}.  
A number of graph classes are known to have bounded nonrepetitive chromatic number. In particular, trees are nonrepetitively 4-colourable \citep{BGKNP-NonRepTree-DM07,KP-DM08}, outerplanar graphs are nonrepetitively $12$-colourable \citep{KP-DM08,BV-NonRepVertex07}, and more generally, every graph with treewidth $k$ is nonrepetitively $4^k$-colourable \citep{KP-DM08}. Graphs with maximum degree $\Delta$ are nonrepetitively \Oh{\Delta^2}-colourable \citep{AGHR-RSA02,Gryczuk-IJMMS07,HJ-DM11,DJKW16}. 

Perhaps the most important open problem in the field of nonrepetitive colourings is whether planar graphs have bounded nonrepetitive chromatic number \citep{AGHR-RSA02}. The best known lower bound is $11$, due to Ochem~\citep{DFJW13}. \citet{DFJW13} showed that layered separations can be used to construct nonrepetitive colourings. In particular, every $n$-vertex graph that admits layered separations of width $\ell$ is nonrepetitively \Oh{\ell\log n}-colourable~\citep{DFJW13}. Applying the result for planar graphs mentioned above, \citet{DFJW13} concluded that every $n$-vertex planar graph is nonrepetitively \Oh{\log n}-colourable. We generalise this result to conclude that every graph with Euler genus $g$ is nonrepetitively \Oh{g+\log n}-colourable (\thmref{GenusNonRep}). The previous best bound for graphs of bounded genus was \Oh{\sqrt{n}}, which is obtained by an easy application of the standard \Oh{\sqrt{n}} separator result for graphs of bounded genus. We further generalise this result to conclude a \Oh{\log n} bound for graphs excluding a fixed topological minor (\thmref{TopoMinorNonRep}).

All our results about nonrepetitive graph colouring are proved in \secref{NonRep}. 

\section{Treewidth and Layered Treewidth}
\seclabel{DecompSep}

Graphs decompositions, especially tree decompositions, are a key to our results. For graphs $G$ and $H$, an $H$-\emph{decomposition} of  $G$ is a collection $(B_x\subseteq V(G):x\in V(H))$ of sets of vertices in $G$ (called  \emph{bags}) indexed by the vertices of $H$, such that:
\begin{enumerate}[(1)]
\item for every edge $vw$ of $G$, some bag $B_x$ contains both $v$ and $w$, and 
\item for every vertex $v$ of $G$, the set $\{x\in V(H):v\in B_x\}$ induces a non-empty connected subgraph of $H$.
\end{enumerate}
The \emph{width} of a decomposition is the size of the largest bag minus 1. If $H$ is a tree, then an $H$-decomposition is called a \emph{tree decomposition}. The \emph{treewidth} of a graph $G$ is the minimum width of any tree decomposition of $G$. Tree decompositions were first introduced by \citet{Halin76} and independently by \citet{RS-GraphMinorsII-JAlg86}. $H$-decompositions, for general graphs $H$, were introduced by \citet{DK05}; also see \citep{WT07}. 

Separations and treewidth are closely connected, as shown by the following two results. 

\begin{lem}[\citep{RS-GraphMinorsII-JAlg86}, (2.5) \& (2.6)]
\lemlabel{RS}
If $S$ is a set of vertices in a graph $G$, then for every tree decomposition of $G$ there is a bag $B$ such that each connected component of $G-B$ contains at most $\frac{1}{2}|S|$ vertices in $S$, which implies that  $G$ has a separation $(G_1,G_2)$ with $V(G_1 \cap G_2)=B$ and both $V(G_1) \setminus V(G_2)$ and $V(G_2) \setminus V(G_1)$ contain at most $\frac{2}{3}|S|$ vertices in $S$.
\end{lem}

\begin{lem}[\citet{Reed97}, Fact~2.7]
\lemlabel{SepToTree}
Assume that for every set $S$ of vertices in a graph $G$, there is a separation $(G_1,G_2)$ of $G$ such that $|V(G_1 \cap G_2)|\leq k$ and both $V(G_1) \setminus V(G_2)$ and $V(G_2) \setminus V(G_1)$ contain at most $\frac{2}{3}|S|$ vertices in $S$. Then $G$ has treewidth less than $4k$. 
\end{lem}

We now define the \emph{layered width} of a decomposition, which is the key original definition of this paper. The \emph{layered width} of an $H$-decomposition $(B_x:x\in V(H))$ of a graph $G$ is the minimum integer $\ell$ such that, for some layering $(V_0,V_1,\dots,V_t)$ of $G$, each bag $B_x$ contains at most $\ell$ vertices in each layer $V_i$. The \emph{layered treewidth} of a graph $G$ is the minimum layered width of a tree decomposition of $G$. 
Layerings with one layer show that layered treewidth is at most treewidth plus 1.

The following result, which is implied by \lemref{RS}, shows that bounded layered treewidth leads to layered separations of bounded width; see \thmref{Equivalent} for a converse result.

\begin{lem}
\lemlabel{DecompLayering} 
Every graph with layered treewidth $\ell$ admits layered separations of width at most $\ell$. 
\end{lem}

The \emph{diameter} of a connected graph $G$ is the maximum distance of two vertices in $G$. Layered tree decompositions lead to tree decompositions of bounded width for graphs of bounded diameter.

\begin{lem}
\lemlabel{TreewdithDiameter}
If a connected graph $G$ has diameter $d$, treewidth $k$ and layered treewidth $\ell$, then 
$k< \ell(d+1)$. 
\end{lem}

\preproof
\begin{proof}
Every layering of $G$ has at most $d+1$ layers. Thus each bag in a tree decomposition of layered width $\ell$ contains at most $\ell(d+1)$ vertices. The claim follows.
\end{proof}

Similarly, a graph of bounded diameter that admits layered separations of bounded width has bounded treewidth. 

\begin{lem}
\lemlabel{LayeredSepDiameter}
If a connected graph $G$ has diameter $d$, treewidth $k$ and admits layered separations of width $\ell$, then 
$k< 4\ell(d+1)$. 
\end{lem}

\preproof
\begin{proof}
Since $G$ admits layered separations of width $\ell$, there is a layering of $G$ such that  for every set $S\subseteq V(G)$, there is a separation $(G_1,G_2)$  of $G$ such that  each layer contains at most $\ell$ vertices in $V(G_1 \cap G_2)$, and  both $V(G_1) \setminus V(G_2)$ and $V(G_2) \setminus V(G_1)$ contain at most $\frac{2}{3}|S|$ vertices in $S$. Since $G$ has diameter $d$, the number of layers is at most $d+1$. Thus  $|V(G_1 \cap G_2)|\leq (d+1)\ell$. The claim follows from \lemref{SepToTree}.
\end{proof}

\twolemref{TreewdithDiameter}{LayeredSepDiameter} can essentially be rewritten in the language of `local treewidth', which was first introduced by \citet{Eppstein-Algo00} under the guise of the `treewidth-diameter' property. A graph class $\mathcal{G}$ has \emph{bounded local treewidth} if there is a function $f$ such that for every graph $G$ in $\mathcal{G}$, for every vertex $v$ of $G$ and for every integer $r\geq0$, the subgraph of $G$ induced by the vertices at distance at most $r$ from $v$ has treewidth at most $f(r)$; see \citep{Grohe-Comb03,DH-SJDM04,DH-SODA04,Eppstein-Algo00}. If $f(r)$ is a linear function, then  $\mathcal{G}$ has \emph{linear local treewidth}.

\begin{lem}
\lemlabel{DecompLinearLocal}
If every graph in some class $\mathcal{G}$ has layered treewidth at most $\ell$, then $\mathcal{G}$ has linear local treewidth with $f(r)=\ell(2r+1)-1$. 
\end{lem}

\begin{proof}
Given a vertex $v$ in a graph $G\in\mathcal{G}$, and given an integer $r\geq 0$, let $G'$ be the subgraph of $G$ induced by the set of vertices at distance at most $r$ from $v$. By assumption, $G$ has a tree decomposition of layered width $\ell$ with respect to some layering $(V_0,V_1,\dots,V_t)$. If $v\in V_i$ then $V(G') \subseteq V_{i-r}\cup\dots\cup V_{i+r}$. Thus $G'$ contains at most $(2r+1)\ell$ vertices in each bag. Hence $G'$ has treewidth at most $(2r+1)\ell-1$, and  $\mathcal{G}$ has linear local treewidth. 
\end{proof}

\begin{lem}
\lemlabel{LayerSepLinearLocal}
If every graph in some class $\mathcal{G}$ admits layered separations of width at most $\ell$, then $\mathcal{G}$ has linear local treewidth with $f(r)<4\ell(2r+1)$. 
\end{lem}

\begin{proof}
Given a vertex $v$ in a graph $G\in\mathcal{G}$, and given an integer $r\geq 0$, let $G'$ be the subgraph of $G$ induced by the set of vertices at distance at most $r$ from $v$. By assumption, there is a layering $(V_0,V_1,\dots,V_t)$ of $G$ such that  for every set $S\subseteq V(G)$, there is a separation $(G_1,G_2)$  of $G$ such that  each layer contains at most $\ell$ vertices in $V(G_1 \cap G_2)$, and  both $V(G_1) \setminus V(G_2)$ and $V(G_2) \setminus V(G_1)$ contain at most    $\frac{2}{3}|S|$ vertices in $S$. If $v\in V_i$ then $V(G') \subseteq V_{i-r}\cup\dots\cup V_{i+r}$. Thus $|V(G_1\cap G_2 \cap G')|\leq (2r+1)\ell$. By \lemref{SepToTree}, $G'$ has treewidth less than $4(2r+1)\ell$. The claim follows.
\end{proof}

We conclude this section with a few observations about layered treewidth. 
First we show that graphs with bounded layered treewidth have linearly many edges. 

\begin{lem}
\lemlabel{LayeredTreewidthEdges}
Every $n$-vertex graph $G$ with layered treewidth $k$ has at most $(3k-1)n$ edges.
\end{lem}

\begin{proof}
We proceed by induction on $n$. The base case is trivial. Let $S$ be a leaf bag in a tree decomposition of $G$ with layered width $k$. Let $T$ be the neighbouring bag. If $S\subseteq T$ then delete $S$ and repeat. Otherwise there is a vertex $v$ in $S\setminus T$. Say $v$ is in layer $V_i$. Then every neighbour of $v$ is in $S\cap (V_{i-1}\cup V_i\cup V_{i+1})\setminus\{v\}$, which has size at most $3k-1$. Thus $G$ has minimum degree at most $3k-1$. Since every subgraph of $G$ has layered treewidth at most $k$,  by induction, $G$ has at most $(3k-1)n$ edges. 
\end{proof}

The following example shows that this bound is roughly tight. For integers $p\gg k\geq 2$, let $G$ be the graph with vertex set $\{(x,y):x,y\in\{1,\dots,p\}\}$, where distinct vertices $(x,y)$ and $(x',y')$ are adjacent if $|y-y'|\leq 1$ and $|x-x'|\leq k-1$. For $y\in\{1,\dots,p\}$, let $V_y:=\{(x,y):x\in\{1,\dots,p\}\}$. Then $(V_1,V_2,\dots,V_p)$ is a layering of $G$. For 
$x\in\{1,\dots,p-k+1\}$, let $B_x:=\{(x',y):x'\in\{x,\dots,x+k-1\},y\in\{1,\dots,p\}\}$. Then $B_1,B_2,\dots,B_{p-k+1}$ is a tree decomposition of $G$ with layered width $k$. Apart from vertices near the boundary, every vertex of $G$ has degree $6k-4$. It follows that $|E(G)|=(3k-2)n-\Oh{k\sqrt{n}}$.

Note that layered treewidth is not a minor-closed parameter. For example, if $G$ is the 3-dimensional $n\times n\times 2$ grid graph, then $G$ has layered treewidth at most 3 (since the $n\times 2$ grid has a tree decomposition with bags of size 3), but $G$ contains a $K_n$ minor \citep{Wood-ProductMinor}, and $K_n$ has layered treewidth $\ceil{\frac{n}{2}}$. On the other hand, we now show that for graphs with bounded layered treewidth, the minors of bounded depth have bounded layered treewidth. 

\begin{lem}
\lemlabel{ShallowMinor}
If $G$ is a graph with layered treewidth $k$, and $H_1,\dots,H_p$ are pairwise disjoint connected subgraphs of $G$, each with radius at most some positive integer $d$, and $G'$ is the graph obtained from $G$ by contracting each $H_i$ into a single vertex, then $G'$ has layered treewidth at most $(4d+1)k$. 
\end{lem}

\begin{proof}
By definition, $G$ has a layering $(V_0,\dots,V_t)$ and a tree decomposition $\mathcal{T}$, such that each bag of $\mathcal{T}$ has at most $k$ vertices in each layer $V_i$. We may assume that $V(G)=\bigcup_iV(H_i)$ (by introducing subgraphs with one vertex). Each subgraph $H_i$ contains a vertex $v_i$ such that every vertex in $H_i$ is at distance at most $d$ from $v_i$ (in $H_i$). We can and do think of $V(G')=\{v_1,v_2,\dots,v_p\}$, where $v_iv_j\in E(G')$ if and only if some vertex in $H_i$ is adjacent to some vertex in $H_j$. In this case, $\dist_G(v_i,v_j)\leq 2d+1$.  Let $t':=\floor{t/(2d+1)}$. For $\ell\in\{0,1,\dots,t'\}$, let 
$$V'_\ell:=V(G')\cap (V_{\ell(2d+1)}\cup V_{\ell(2d+1)+1}\cup\dots\cup V_{(\ell+1)(2d+1)-1}),$$
where $V_j:=\emptyset$ for $j>t$. Then $(V'_0,\dots,V'_{t'})$ is a partition of $V(G')$. If $v_iv_j\in E(G')$ and $v_i\in V_a$ and $v_j\in V_b$, then $|b-a|\leq \dist_G(v_i,v_j)\leq 2d+1$. It follows that if $v_i\in V'_{a'}$ and $v_j\in V'_{b'}$, then $|a'-b'|\leq 1$. Hence $(V'_0,\dots,V'_{t'})$ is a layering of $G'$. 

Let $\mathcal{T'}$ be the tree decomposition of $G'$ obtained from $\mathcal{T}$ by replacing each bag $B$ of $T$ by a new bag $B'$ consisting of each vertex $v_i$ of $G'$ for which $H_i$ contains a vertex in $B$. Consider a vertex $v_i$ in $V'_\ell\cap B'$ for some layer $V'_\ell$ and bag $B'$ of $\mathcal{T'}$. 
Thus $H_i$ contains a vertex $w$ in $B$. Since $v_i\in V'_\ell$ and $H_i$ has radius at most $d$, in the original layering, $w$ is in $V_{\ell(2d+1)-d}\cup V_{\ell(2d+1)-d+1}\cup\dots\cup V_{(\ell+1)(2d+1)+d-1}$. 
There are at most $(4d+1)k$ such vertices $w$ in $B$. 
Thus  $|V'_\ell \cap B'|\leq (4d+1)k$, and $G'$ has layered treewidth at most $(4d+1)k$.
%
%
%
%
%
\end{proof}

\twolemref{LayeredTreewidthEdges}{ShallowMinor} together show that graphs with bounded layered treewidth have bounded expansion; see \cite{Sparsity}. 

The following result, due to Sergey Norin [personal communication, 2014], shows that graphs with bounded layered treewidth have \Oh{\sqrt{n}} treewidth.

\begin{lem}
\lemlabel{Norine}
Every $n$-vertex graph $G$ with layered treewidth $k$ has treewidth at most $2\sqrt{kn}-1$.
\end{lem}

\begin{proof}
Let $(V_1,V_2,\dots,V_t)$ be the layering in a tree decomposition of $G$ with layered width $k$. Let $p:=\ceil{\sqrt{n/k}}$.  For $j\in\{1,\dots,p\}$, let $W_j:=V_j\cup V_{p+j}\cup V_{2p+j}\cup\cdots$. Thus $(W_1,W_2,\dots,W_p)$ is a partition of $V(G)$, and $|W_j|\leq \frac{n}{p} \leq \sqrt{kn}$ for some $j\in\{1,\dots,p\}$. Each connected component of $G-W_j$ is contained within $p-1$ consecutive layers, and therefore has treewidth at most $k(p-1)-1\leq \sqrt{kn}-1$. Hence $G-W_j$ has a tree decomposition of width at most $\sqrt{kn}-1$. Adding $W_j$ to every bag of this decomposition gives a tree decomposition of $G$ with width at most $\sqrt{kn}-1+|W_j|\leq 2\sqrt{kn}-1$. 
\end{proof}

\section{Graphs on Surfaces}
\seclabel{Surfaces}

This section constructs layered tree decompositions of graphs with bounded Euler genus. The following definitions and simple lemma will be useful. A \emph{triangulation} of a surface is a loopless multigraph embedded in the surface, such that each face is bounded by three distinct edges. We emphasise that parallel edges not bounding a single face are allowed. For a subgraph $G'$ of $G$, let  $F(G')$ be the set of faces of $G$ incident with at least one vertex of $G'$. Let $G^*$ be the dual of $G$. That is, $V(G^*)=F(G)$ and $fg\in E(G^*)$ whenever some edge of $G$ is incident with both $f$ and $g$ (for all distinct faces $f,g\in F(G)$). Thus the edges of $G$ are in 1--1 correspondence with the edges of $G^*$. Let $T$ be a subtree of $G$. An edge $vw\in E(G)$ is a \emph{chord} of $T$ if $v,w\in V(T)$ and $vw\not\in E(T)$. An edge $vw\in E(G)$ is a \emph{half-chord} of $T$ if $|\{v,w\}\cap V(T)|=1$. An edge of $G^*$ dual to a chord of $G$ is called a \emph{dual-chord}. An edge of $G^*$ dual to a half-chord of $G$ is called a \emph{dual-half-chord}. 

\begin{lem}
\lemlabel{AroundTree}
Let $T$ be a non-empty subtree of a triangulation $G$ of a surface. Let $H$ be the subgraph of $G^*$ with vertex set $F(T)$ and edge set the dual-chords and dual-half-chords of $T$. Then $H$ is connected. Moreover, $H-e$ is connected for each dual-half-chord $e$ of $T$. 
\end{lem}

\begin{proof}
If $T$ has exactly one vertex $v$, then $T$ has no chords, and the half-chords of $T$ are precisely the edges incident to $v$, in which case $H$ is a cycle on at least two vertices, and the result is trivial. Now assume that $|V(T)|\geq 2$ and thus $|E(T)|\geq 1$. 

Consider the following walk $W$ in $T$, illustrated in \figref{TreeWalk}. Choose an arbitrary edge $\alpha\beta$ in $T$, and initialise $W:=(\alpha,\beta)$. Apply the following rule to choose the next vertex in $W$. Suppose that $W=(\alpha,\beta,\dots,x,y)$. Let $yz$ be the edge of $T$ anticlockwise from $yx$ in the cyclic permutation of edges incident to $y$ defined by the embedding of $T$. (It is possible that $x=z$.)\ Then append $z$ to $W$. Stop when the edge $\alpha\beta$ is traversed in this order for the second time. Thus each edge of $T$ is traversed by $W$ exactly two times (once in each direction), and $W$ is a closed (cyclic) walk. 

Let $W'$ be the walk in $H$ obtained from $W$ as follows. Consider three consecutive vertices $x,y,z$ in $W$. Let $f_1,f_2,\dots,f_{k}$ be the  sequence of faces anticlockwise from $yx$ to $yz$ determined by the cyclic permutation of edges incident with $y$. Construct $W'$ from $W$ by replacing $y$ by $f_1,f_2,\dots,f_{k-1}$ (and doing this simultaneously at each vertex in $W$). Each such face $f_i$ is incident with $y$, and is  thus a vertex of $H$. Moreover, for $i\in\{1,\dots,k-1\}$, the edge $f_if_{i+1}$ of $G^*$ is dual to a chord or half-chord of $T$, and thus $f_if_{i+1}$ is an edge of $H$. Hence $W'$ is a walk in $H$ (since $f_k$ is the first face in the sequence of faces corresponding to $z$). Every face of $G$ incident with at least one vertex in $T$ appears in $W'$. Thus $W'$ is a spanning walk in $H$. Therefore $H$ is connected, as claimed. 

\begin{figure}[!ht]
\begin{center}
\includegraphics{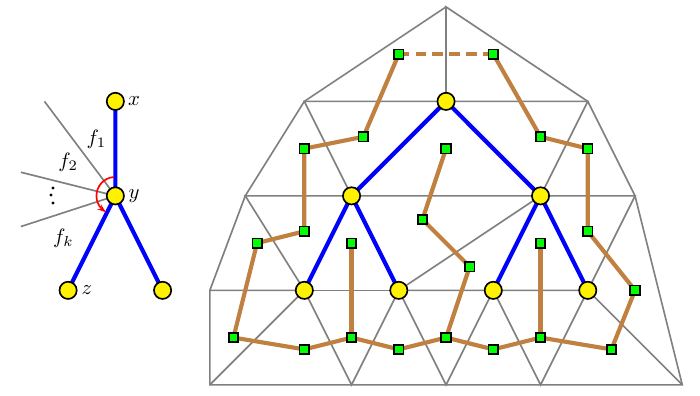}
\caption{Construction of $H$ in \lemref{AroundTree}.}
\figlabel{TreeWalk}
\end{center}
\end{figure}

Let $H'$ be the subgraph of $H$ formed by the dual-half-chords of $T$. We now show that $H'$ is 2-regular. Consider a dual-half-chord $fg$ of $T$. Let $vw$ be the corresponding half-chord of $G$, where $v\in V(T)$ and $w\not\in V(T)$. Say $u$ is the third vertex incident to $f$. If $u\in V(T)$ then $uv$ is not a half-chord of $T$ and $uw$ is a half-chord of $T$, implying that the only edges incident to $f$ in $H'$ are the duals of $vw$ and $uw$. On the other hand, if $u\not\in V(T)$ then $uv$ is a half-chord of $T$ and $uw$ is not a half-chord of $T$, implying that the only edges incident to $f$ in $H'$ are the duals of $vw$ and $uv$. Hence $f$ has degree 2 in $H'$, and $H'$ is 2-regular. Therefore, if $e$ is a dual-half-chord of $T$, then $e$ is in a cycle, and $H-e$ is connected. 
%
%
\end{proof}

The following theorem is the main result of this section. If $v$ is a vertex in a tree $T$ rooted at a vertex $r$, then the \emph{subtree of $T$ rooted at $v$} is the subtree of $T$ induced by the set of vertices $x$ in $T$ such that $v$ is on the $xr$-path in $T$. 

\begin{thm} 
\thmlabel{GenusDecomposition}
Every graph $G$ with Euler genus $g$ has layered treewidth at most $2g+3$. 
\end{thm}

\preproof\begin{proof} 
Say $G$ has  $n$ vertices.  We may assume that $n\geq 3$ and that $G$ is a triangulation of a surface with Euler genus $g$. Let $F(G)$ be the set of faces of $G$. By Euler's formula, $|F(G)|=2n+2g-4$ and $|E(G)|=3n+3g-6$. Let $r$ be a vertex of $G$. Let $(V_0,V_1,\dots,V_t)$ be the bfs layering of $G$ starting from $r$. Let $T$ be a bfs tree of $G$ rooted at $r$. For each vertex $v$ of $G$, let $P_v$ be the vertex set of the $vr$-path in $T$. Thus if $v\in V_i$, then  $P_v$ contains exactly one vertex in  $V_j$ for $j\in\{0,\dots,i\}$.

Let $D$ be the subgraph of $G^*$ with vertex set $F(G)$, where two vertices are adjacent if the corresponding faces share an edge not in $T$. Thus $|V(D)|=|F(G)|=2n+2g-4$ and $|E(D)|=|E(G)|-|E(T)|=(3n+3g-6)-(n-1)=2n+3g-5$. 
Since $V(T)=V(G)$, each edge of $G$ is either an edge of $T$ or is a chord of $T$. Thus $D$ is the graph $H$ defined in \lemref{AroundTree}. By \lemref{AroundTree}, $D$ is connected. 

Let $T^*$ be a spanning tree of $D$. Thus $|E(T^*)|=|V(D)|-1=2n+2g-5$. Let $X^*:=E(D)\setminus E(T^*)$ and let $X$ be the set of edges of $G$ dual to the edges in $X^*$. 
Thus $|X|=|X^*|=(2n+3g-5)-(2n+2g-5)=g$. For each face $f=xyz$ of $G$, let 
$$C_f:= \bigcup \{P_a\cup P_b:ab\in X\} \cup P_x\cup P_y\cup P_z\enspace.$$ 
Since $|X|=g$ and each $P_v$ contains at most one vertex in each layer,  $C_f$ contains at most $2g+3$ vertices in each layer.

We claim that $(C_f:f\in F(G))$ is a $T^*$-decomposition of $G$. For each edge $vw$ of $G$, if $f$ is a face incident to $vw$ then $v$ and $w$ are in $C_f$. This proves condition (1) in the definition of $T^*$-decomposition.

We now prove condition (2). It suffices to show that  for each vertex $v$ of $G$, if $F'$ is the set of faces $f$ of $G$ such that $v$ is in $C_f$, then the induced subgraph $T^*[F']$ is connected and non-empty. 
Each face incident to $v$ is in $F'$, thus $F'$ is non-empty. Let $T'$ be the subtree of $T$ rooted at $v$. 
If some edge $ab$ in $X$ is a half-chord or chord of $T'$, then  $v$ is in $P_a\cup P_b$, implying that $v$ is in every bag, and $T^*[F']=T^*$ is connected. Now assume that no half-chord or chord of $T'$ is in $X$. 
Thus a face $f$ of $G$ is in $F'$ if and only if $f$ is incident with a vertex in $T'$; that is, $F'=F(T')$. 
If $v=r$, then $T'=T$ and $F'=F(G)$, implying $T^*[F']=T^*$, which is connected. 
Now assume that $v\neq r$. Let $p$ be the parent of $v$ in $T$. 
Let $H$ be the graph defined in \lemref{AroundTree} with respect to $T'$. 
So $H$ has vertex set $F'$ and edge set the dual-chords and dual-half-chords of $T'$. Each chord or half-chord of $T'$ is an edge of $G-(E(T)\cup X)$, except for  $pv$, which is a half-chord of $T'$ (since $p\not\in V(T')$). Let $e$ be the edge of $H$ dual to $pv$. By \lemref{AroundTree}, $T^*[F']=H-e$  is connected, as desired. 

Therefore $(C_f:f\in F(G))$ is a $T^*$-decomposition of $G$ with layered width at most $2g+3$. 
\end{proof}

Several notes on \thmref{GenusDecomposition} are in order. 

\begin{itemize}

\item A spanning tree in an embedded graph with an `interdigitating' spanning tree in the dual was introduced for planar graphs by \citet{vonStaudt} in 1847, and is sometimes called a \emph{tree-cotree decomposition} \citep{Eppstein-SODA03}. This idea was generalised for  orientable surfaces by \citet{Biggs71} and for non-orientable surfaces by \citet{RiSh84}; also see \citep{Skoviera92}. 

\item \lemref{DecompLayering} and \thmref{GenusDecomposition} imply the following result for layered separators. 

\begin{thm} 
\thmlabel{GenusLayering}
Every graph with Euler genus $g$ admits layered separations of width $2g+3$.
\end{thm}

\lemref{Norine} and  \thmref{GenusDecomposition} imply the following bound on treewidth:

\begin{thm}
\thmlabel{GenusTreewidth}
Every $n$-vertex graph with Euler genus $g$ has treewidth at most $2\sqrt{(2g+3)n}-1$.
\end{thm}

\lemref{RS} then implies that $n$-vertex graphs of Euler genus $g$ have separators of order  \Oh{\sqrt{gn}}, as proved in  \citep{Eppstein-SODA03,GHT-JAlg84,AS-SJDM96,Djidjev87}. \citet{GHT-JAlg84} gave examples of such graphs with no $o(\sqrt{gn})$ separator, and thus with treewidth $\Omega(\sqrt{gn})$ by \lemref{RS}. Hence each of the upper bounds in \thmref{GenusDecomposition}--\ref{thm:GenusTreewidth} are within a constant factor of optimal. 

Note that the proof of \thmref{GenusDecomposition} uses ideas from many previous proofs about separators in embedded graphs  \citep{Eppstein-SODA03,GHT-JAlg84,AS-SJDM96}. For example, \citet{AS-SJDM96} call the graph $D$ in the proof of \thmref{GenusDecomposition} a \emph{separation graph}. 

\item If we apply \thmref{GenusDecomposition} to a graph with radius $d$, where  $r$ is a central vertex, then each bag consists of $2g+3$ paths ending at $r$, each of length at most $d$. Thus each bag contains at most $(2g+3)d+1$ vertices. We obtain the following result, first proved in the planar case by \citet{RS-GraphMinorsIII-JCTB84} and implicitly  by \citet{Baker94}, and in general by \citet{Eppstein-Algo00} with a \Oh{gd} bound. Eppstein's proof also uses the  tree-cotree decomposition; see \citep{Eppstein99,Eppstein-SODA03} for related work. 

\begin{thm}
Every graph with Euler genus $g$ and radius $d$ has treewidth at most $(2g+3)d$. In particular, every planar graph with radius $d$ has treewidth at most $3d$. 
\end{thm}

\item The proof of \thmref{GenusDecomposition} gives the following stronger result that will be useful later, where $Q=\bigcup\{P_a\cup P_b:ab\in X\}$. 

\begin{thm}
\thmlabel{NewGenusDecomposition}
Let $r$ be a vertex in a graph $G$ with Euler genus $g$.  Then there is a tree decomposition $\mathcal{T}$ of $G$ with layered width at most $2g+3$ with respect to some layering in which the first layer is $\{r\}$. Moreover, there is a set $Q\subseteq V(G)$ with at most $2g$ vertices in each layer, such that $\mathcal{T}$ restricted to $G-Q$ has layered width at most 3 with respect to the same layering. 
\end{thm}

\end{itemize}

\section{Clique-Sums}

We now extend the above results to more general graph classes via the clique-sum operation. For compatibility with this operation, we introduce the following concept that is slightly stronger than having bounded layered treewidth. A \emph{clique} is a set of pairwise adjacent vertices in a graph. Say a graph $G$ is \emph{$\ell$-good} if for every clique $K$ of size at most $\ell$ in $G$ there is a tree decomposition of $G$ of layered width at most $\ell$ with respect to some layering of $G$ in which $K$ is the first layer.

\begin{thm} 
\thmlabel{GoodDecomposition}
Every graph $G$ with Euler genus $g$ is $(2g+3)$-good.
\end{thm}

\begin{proof}
Given a clique $K$ of size at most $2g+3$ in $G$, let $G'$ be the graph obtained from $G$ by contracting $K$ into a single vertex $r$. Then $G'$ has Euler genus at most $g$. \thmref{NewGenusDecomposition} gives a tree decomposition of $G'$ of layered width at most $2g+3$ with respect to some layering of $G'$ in which $\{r\}$ is the first layer. Replace the first layer by $K$, and replace each instance of $r$ in the tree decomposition of $G'$ by $K$. We obtain a tree decomposition of $G$ of layered width at most $2g+3$ with respect to some layering of $G$ in which $K$ is the first layer (since $|K|\leq 2g+3$). Thus $G$ is $(2g+3)$-good. 
\end{proof}

Let $C_1=\{v_1,\dots,v_k\}$ be a $k$-clique in a graph $G_1$. Let $C_2=\{w_1,\dots,w_k\}$ be a $k$-clique in a graph $G_2$. Let $G$ be the graph obtained from the disjoint union of $G_1$ and $G_2$ by identifying $v_i$ and $w_i$ for $i\in\{1,\dots,k\}$, and possibly deleting some edges in $C_1$ ($=C_2$). Then $G$ is  a \emph{$k$-clique-sum} of $G_1$ and $G_2$. If $k\leq \ell$ then $G$ is  a \emph{$(\leq\ell)$-clique-sum} of $G_1$ and $G_2$

\begin{lem}
\lemlabel{GoodCliqueSum}
For $\ell\geq k$, if $G$ is a $(\leq k)$-clique-sum of $\ell$-good graphs $G_1$ and $G_2$, then $G$ is $\ell$-good.
\end{lem}

\preproof\begin{proof} 
Let $K$ be a clique of size at most $\ell$ in $G$. Without loss of generality, $K$ is in $G_1$. 
Since $G_1$ is $\ell$-good, there is a tree decomposition $T_1$ of $G_1$ of layered width at most $\ell$ with respect to some layering of $G_1$ in which $K$ is the first layer. Let $X:=V(G_1\cap G_2)$. Thus $X$ is a clique in $G_1$ and in $G_2$. Hence $X$ is contained in at most two consecutive layers of the above layering of $G_1$. Let $X'$ be the subset of $X$ in the first of these two layers. Note that if $K\cap X\neq\emptyset$ then $X'=K\cap X$. Since $|X'|\leq k\leq\ell$ and since $G_2$ is $\ell$-good, there is a tree decomposition $T_2$ of $G_2$ with layered width at most $\ell$ with respect to some layering of $G_2$ in which $X'$ is the first layer. Thus the second layer of $G_2$ contains $X\setminus X'$. Now, the layerings of $G_1$ and $G_2$ can be overlaid, with the layer containing $X'$ in common, and the layer containing $X\setminus X'$ in common. By the definition of $X'$, it is still the case that the first layer is $K$. Let $T$ be the tree decomposition of $G$ obtained from the disjoint union of $T_1$ and $T_2$ by adding an edge between a bag in $T_1$ containing $X$ and a bag in $T_2$ containing $X$. (Each clique is contained in some bag  of a tree decomposition.)\ For each bag $B$ of $T$ the intersection of $B$ with a single layer  consists of  the same set of vertices as the intersection of $B$ and the corresponding layer in the layering of $G_1$ or $G_2$. Hence  $T$ has layered width at most $\ell$. 
\end{proof}

We now describe some graph classes for which \lemref{GoodCliqueSum} is immediately applicable. \citet{Wagner37} proved that every $K_5$-minor-free graph can be constructed from $(\leq 3)$-clique-sums of planar graphs and $V_8$, where $V_8$ is the graph obtained from an 8-cycle by 
adding four edges between the opposite pairs of vertices. A bfs layering shows that $V_8$ is 3-good. By \thmref{GoodDecomposition}, every planar graph is 3-good. Thus, by \lemref{GoodCliqueSum}, every $K_5$-minor-free graph is $3$-good, has layered treewidth at most 3, and admits layered separations of width 3 by \lemref{DecompLayering}.  \citet{Wagner37} and \citet{Hall43} also proved that every $K_{3,3}$-minor-free graph can be constructed from $(\leq 2)$-clique-sums of planar graphs and $K_5$. Since $K_5$ is $4$-good and every planar graph is 3-good, every $K_{3,3}$-minor-free graph is $4$-good, has layered treewidth at most 4, and admits layered separations of width 4. For a number of particular graphs $H$, \citet{Truemper} characterised the $H$-minor-free graphs in terms of $(\leq3)$-clique-sums of planar graphs and various small graphs. The above methods apply here also; we omit these details. More generally, a graph $H$ is \emph{single-crossing} if it has a drawing in the plane with at most one crossing. For example, $K_5$ and $K_{3,3}$ are single-crossing. \citet{RS-GST93} proved that for every single-crossing graph $H$, every $H$-minor-free graph can be constructed from $(\leq 3)$-clique-sums of planar graphs and graphs of treewidth at most $\ell$, for some constant $\ell=\ell(H)\geq 3$. It follows from the above results that every $H$-minor-free graph is $\ell$-good, has layered treewidth at most $\ell$, and admits layered separations of width $\ell$. 

\section{The Graph Minor Structure Theorem}
\seclabel{Vortices}

This section introduces the graph minor structure theorem of Robertson and Seymour. This theorem shows that every graph in a proper minor-closed class can be constructed using four ingredients: graphs on surfaces, vortices, apex vertices, and clique-sums. We show that, with a restriction on the apex vertices, every graph that can be constructed using these ingredients has bounded layered treewidth, and thus admits layered separations of bounded width. 

Let $G_0$ be a graph embedded in a surface $\Sigma$. Let $F$ be a facial cycle of $G_0$ (thought of as a subgraph of $G_0$). An  \emph{$F$-vortex} is an $F$-decomposition $(B_x\subseteq V(H):x\in V(F))$ of a graph $H$ such that $V(G_0\cap H)=V(F)$ and $x\in B_x$ for each $x\in V(F)$.  For $g,p,a\geq0$ and $k\geq1$, a graph $G$ is \emph{$(g,p,k,a)$-almost-embeddable} if for some set $A\subseteq V(G)$ with $|A|\leq a$, there are graphs $G_0,G_1,\dots,G_s$ for some $s\in\{0,\dots,p\}$ such that:
\begin{itemize}
\item $G-A = G_{0} \cup G_{1} \cup \cdots \cup G_s$, 
\item $G_{1}, \dots, G_s$ are pairwise vertex-disjoint;
\item $G_{0}$ is embedded in a surface of Euler genus at most $g$, 
\item there are $s$ pairwise vertex-disjoint facial cycles $F_1,\dots,F_s$ of $G_0$, and
\item for $i\in\{1,\dots,s\}$, there is an $F_i$-vortex $(B_x\subseteq V(G_i):x\in V(F_i))$ of  $G_i$ of width at most $k$. 
\end{itemize}
The vertices in $A$ are called \emph{apex} vertices. They can be adjacent to any vertex in $G$. 

A graph is \emph{$k$-almost-embeddable} if it is  $(k,k,k,k)$-almost-embeddable. The following graph minor structure theorem by Robertson and Seymour is at the heart of  graph minor theory. In a tree decomposition $(B_x\subseteq V(G):x\in V(T))$ of a graph $G$, the \emph{torso} of a bag $B_x$ is the subgraph obtained from $G[B_x]$ by adding all edges $vw$ where $v,w\in B_x\cap B_y$ for some edge $xy\in E(T)$. 

\begin{thm}[\citet{RS-GraphMinorsXVI-JCTB03}]
\thmlabel{GMST}
For every fixed graph $H$ there is a constant $k=k(H)$ such that every $H$-minor-free graph is obtained by clique-sums of $k$-almost-embeddable graphs. Alternatively, every $H$-minor-free graph has a tree decomposition in which each torso is $k$-almost-embeddable. 
\end{thm}

This section explores which graphs described by the graph minor structure theorem admit layered separations of bounded width. As stated earlier, it is not the case that all such graphs admit layered separations of bounded width. For example, let $G$ be the graph obtained from the $\sqrt{n}\times\sqrt{n}$ grid by adding one dominant vertex. Thus $G$ has diameter $2$, contains no $K_6$-minor, and has treewidth at least $\sqrt{n}$. By \lemref{LayeredSepDiameter}, if $G$ admits layered separations of width $\ell$, then $\ell\in \Omega(\sqrt{n})$.

We will show that the following restriction to the definition of almost-embeddable will lead to graph classes that admit  layered separations of bounded width. A graph $G$ is \emph{strongly $(g,p,k,a)$-almost-embeddable} if it is $(g,p,k,a)$-almost-embeddable and there is no edge between an apex vertex and a vertex in $G_0-(G_1\cup\dots\cup G_s)$. That is, each apex vertex is only adjacent to other apex vertices or vertices in the vortices. A graph is \emph{strongly $k$-almost-embeddable} if it is strongly $(k,k,k,k)$-almost-embeddable.

\begin{thm}
\thmlabel{StrongGood}
Every strongly $(g,p,k,a)$-almost-embeddable graph $G$ is $(a+(k+1)(2g+2p+3))$-good.
\end{thm}

\preproof\begin{proof}
We use the notation from the definition of strongly $(g,p,k,a)$-almost-embeddable. We may assume that $G$ is connected, $|V(G_0)|\geq 3$, and except for $F_1,\dots,F_s$, each face of $G_0$ is a triangle, where $G_0$ might contain parallel edges not bounding a single face. If $s=0$ then $G$ has no vortices and thus has no apex vertices (since apex vertices only attach to vortices), in which case $G$ is $(g,0,0,0)$-almost-embeddable and thus has Euler genus $g$, and the result follows from \thmref{GoodDecomposition}. 

Let $K$ be a clique in $G$ of size at most $a+(k+1)(2g+2p+3)$.

Construct a layering $(V_0,V_1,\dots,V_t)$ of $G$ as follows. Let $V_0:=K$ and let 
$$V_1:=(N_G(K)\cup A\cup V(G_1\cup\dots\cup G_s))\setminus K\enspace.$$ 
For $i=2,3,\dots$, let $V_i$ be the set of vertices of $G$ that are not in $V_0\cup\dots\cup V_{i-1}$ and are adjacent to some vertex in $V_{i-1}$. Thus $(V_0,V_1,\dots,V_t)$ is a layering of $G$ for some $t$. 

Let $K':= (K\cap V(G_0))\setminus V(F_1\cup\dots\cup F_s)$ be the part of $K$ embedded in the surface and avoiding the vortices. If $K'\neq\emptyset$ then let $r$ be one vertex in $K'$, otherwise $r$ is undefined. 

Let  $G'_0$ be the triangulation obtained from $G_0$ as follows. For $i\in\{1,\dots,s\}$, add a new vertex $r_i$ inside  face $F_i$ (corresponding to vortex $G_i$) and add an edge between $r_i$ and each vertex of $F_i$. Let $n:=|V(G'_0)|$.

We now construct a spanning forest $T$ of $G'_0$. Declare $r$ (if defined) and $r_1,\dots,r_s$ to be the roots of $T$. For $i\in\{1,\dots,s\}$, make each vertex in $V(F_i)$ adjacent to $r_i$ in $T$. By definition, these edges are in $G'_0$. Now, make each vertex in $K'\setminus\{r\}$ adjacent to $r$ in $T$. Since $K'$ is a clique, these edges are in $G'_0$. Note that every vertex in $K\cap V(G'_0)$ is now in $T$. 
Every vertex $v$ in $V(G'_0)\cap V_1$ that is not already in $T$ is adjacent to $K\cap V(G'_0)$; make each such vertex $v$ adjacent to a neighbour in $K\cap V(G'_0)$ in $T$. Every vertex in $V(G'_0)\cap V_1$ is now in $T$ (either as a root or as a child or grandchild of a root). Now, for $i=2,3,\dots$, for each vertex $v$ in $V(G'_0)\cap V_i$, choose a neighbour $w$ of $v$ in $V_{i-1}$, and add the edge $vw$ to $T$. Now, $T$ is a spanning forest of $G'_0$ with $s$ or $s+1$ connected components, and thus with $n-s$ or $n-s-1$ edges. 
 
Let $D$ be the graph with vertex set $F(G_0')$ where two vertices of $D$ are adjacent if the corresponding faces share an edge in $G'_0-E(T)$. Since $G'_0$ has $3n+3g-6$ edges and $2n+2g-4$ faces, $|V(D)|=2n+2g-4$ and $|E(D)|=|E(G_0)|-|E(T)|\leq (3n+3g-6)-(n-s-1)=2n+3g+s-5$. 

We now prove that $D$ is connected. Observe that $D$ is the spanning subgraph of the dual of $G_0'$ obtained by deleting edges dual to edges of $T$. The dual of $G_0'$ is connected. Say $e$ is an edge in some component $T_1$ of $T$. Let $f$ and $g$ be the faces of $G_0'$ incident to $e$. Let $H$ be the connected subgraph defined in \lemref{AroundTree} with respect to $T_1$. Observe that $f$ and $g$ are vertices of $H$, and $H$ is a subgraph of $D$. Since $H$ is connected, any path in the dual of $G_0'$ that uses $e$ can be rerouted via an $fg$-path in $H$. Hence $D$ is connected. 

Let $T^*$ be a spanning tree of $D$. Let $X^*:=E(D)\setminus E(T^*)$ and let $X$ be the set of edges in $G_0'$ dual to the edges in $X$. In fact, $X\subseteq E(G_0)$ since $E(G'_0)\setminus E(G_0)\subseteq E(T)$. Note that  $|X|=|X^*|\leq (2n+3g+s-5)-(2n+2g-4-1)=g+s$. 
 
For each vertex $x\in V(G_0)$, let $P_x$ be the path in $T$ between $x$ and the root of the connected component of $T$ containing $x$. By construction, $P_x$ includes at most one vertex in $G_0$ in each layer $V_i$ with $i\geq 1$. If $P_x$ is in the component of $T$ rooted at $r$, then let $P^+_x:=V(P_x)\setminus K$. Otherwise, $P_x$ is in the component of $T$ rooted at  $r_i$ for some $i\in\{1,\dots,s\}$. Then $P_x$ contains exactly one vertex $v\in V(F_i\cap P_x)$.  Let $P_x^+:=(V(P_x)\setminus\{r_i\})\cup B_v$, where $B_v$ is the bag indexed by $v$ in the vortex $G_i$. Thus $P_x^+$ is a set of vertices in $G$ with at most $k+1$ vertices in each layer $V_i$ with $i\geq1$ (since  $|B_v|\leq k+1$). Define $P^+_{r_i}:=\emptyset$ for $i\in\{1,2,\dots,s\}$.  Define
$$S:=\bigcup\{P_x^+\cup P_y^+ : xy \in X \}.$$
Note that $S$ contains at most $2(k+1)(g+s)$ vertices in each layer $V_i$ (since $|X|\leq g+s$). 
For each face $f=uvw$ of $G'_0$, let 
$$C_f:= P_u^+\cup P_v^+\cup P_w^+\cup A\cup K\cup S.$$ 
Thus $C_f$ contains at most $a+(k+1)(2g+2s+3)$ vertices in each layer $V_i$ (since $|K|\leq a+(k+1)(2g+2s+3)$). 
 
We now prove that $(C_f:f\in F(G'_0))$ is a $T^*$-decomposition of $G$. (This makes sense since $V(T^*)=F(G'_0)$.)\ 
First, we prove condition (1) in the definition of $T^*$-decomposition for each edge $vw$ of $G$. If $v\in A\cup K$, then $v$ is in every bag and $w$ is in some bag (proved below), implying $v$ and $w$ are in a common bag. Now assume that $v\not\in A\cup K$ and $w\not\in A\cup K$ by symmetry. If $vw\in E(G_0)$, then $v,w\in C_f$ for each of the two faces $f$ of $G_0'$ incident to $vw$. Otherwise $vw\in E(G_i)$ for some $i\in\{1,\dots,s\}$. Then $v,w\in B_x$ for some vertex $x\in V(F_i)$, implying that $v,w\in C_f$ for each face $f$ of $G_0'$ incident to $x$. This proves condition (1) in the definition of $T^*$-decomposition.
 
We now prove condition (2) in the definition of $T^*$-decomposition for each vertex  $v$ of $G$. Consider the following three cases:

(a) $v\in A\cup K \cup S$: Then $v$ is in every bag, and condition (2) is satisfied for $v$. 
 
(b) $v\in V(G_0)\setminus (A\cup K\cup S\cup V(G_1\cup\dots\cup G_s))$: Let $F'$ be the set of faces $f$ of $G'_0$ such that $v$ is in $C_f$. Each face incident to $v$ is in $F'$, thus $F'$ is non-empty. It now suffices to prove that the induced subgraph $T^*[F']$ is connected. Let $T'$ be the subtree of $T$ rooted at $v$. If some edge $xy$ in $X$ is a half-chord or chord of $T'$, then  $v$ is in $P_x\cup P_y$ and $v\in S$, which is already handled by case (a). Now assume that no half-chord or chord of $T'$ is in $X$. Then a face $f$ of $G'_0$ is in $F'$ if and only if $f$ is incident with a vertex in $T'$; that is, $F'=F(T')$. Let $H$ be the graph defined in \lemref{AroundTree} with respect to $T'$. That is, $H$ has vertex set $F'$ and edge set the dual-chords and dual-half-chords of $T'$. Since $v$ is in $G_0- K$, it follows that $v$ is not a root of $T$. 
Let $p$ be the parent of $v$ in $T$. Each chord or half-chord of $T'$ is an edge of $G-(E(T)\cup X)$, except for $pv$, which is a half-chord of $T'$ (since $p\not\in V(T')$). Let $e$ be the edge of $H$ dual to $pv$. By \lemref{AroundTree},  $T^*[F']=H-e$ is connected, as desired. 
 
(c) $v\in V(G_i) \setminus (A\cup K\cup S )$ for some $i\in\{1,\dots,s\}$: 
Let $F'$ be the set of faces $f$ of $G'_0$ such that $v$ is in $C_f$. It suffices to prove that the induced subgraph $T^*[F']$ is connected and non-empty. 
Let $Z:=\{z\in V(F_i):v\in B_z\}$, where $B_z$ is the bag of $G_i$ corresponding to $z$. 
By the definition of a vortex, $Z$ induces a connected non-empty subgraph of the cycle $F_i$. 
Say $Z=(z_1,z_2,\dots,z_q)$ ordered by $F_i$ where $q\geq 1$.  
For $j\in\{1,\dots,q\}$, let $T_j$ be the subtree of $T$ rooted at $z_j$. 
Let $F'_j$ be the set of faces of $G_0'$ incident to some vertex in $T_j$. 
Since $v\not\in A\cup K\cup S$, by construction, $T^*[F']=\bigcup_j T^*[F'_j]$. 
By the argument used in part (b) applied to $z_j$, $T^*[F'_j]$ is connected and non-empty. 
Since $F'_j$ and $F'_{j+1}$ have the face $r_iz_jz_{j+1}$ in common for $j\in\{1,\dots,q-1\}$, 
it follows that  $T^*[F']=\bigcup_j T^*[F'_j]$ is connected and non-empty, as desired. 
 
Therefore $(C_f:f\in F(G'_0))$ is a $T^*$-decomposition of $G$, and it has  layered width at most $a+(k+1)(2g+2s+3)$. 
\end{proof}

The following fact is well known.

\begin{lem}
\lemlabel{CliqueSize}
Every clique in a $(g,p,k,a)$-almost-embeddable graph has order at most $a+2k+\floor{\half(7+\sqrt{1+24g})}$.
\end{lem}

\begin{proof}
Say $C$ is a clique in a $(g,p,k,a)$-almost-embeddable graph $G$. Let $A,G_0,G_1,\dots,G_p$ be defined as above. 
Then $C\cap V(G_0)$ has Euler genus at most $g$, and by Euler's formula, $|C\cap V(G_0)|\leq \floor{\half(7+\sqrt{1+24g})}$. No vertex in $G_i-G_0$ is adjacent to a vertex in $G_j-G_0$ for distinct $i,j\geq 1$. Thus $C\cap V(G_i-G_0)$ is non-empty for at most one value of $i\geq 1$. Moreover, 
$|C\cap V(G_i-G_0)|\leq 2k$, since deleting one bag from $G_i-G_0$ (which has size $k$) leaves a graph with pathwidth $k-1$, which has maximum clique size $k$. 
Of course, $|C\cap A|\leq |A|=a$. In total, $|C|\leq a+2k+\floor{\half(7+\sqrt{1+24g})}$.
\end{proof}

For $k\geq 1$ and $p\geq 0$, we have $a+2k+\floor{\half(7+\sqrt{1+24g})}\leq a+(k+1)(2g+2p+3)$. Thus \lemref{DecompLayering}, \lemref{GoodCliqueSum}, \thmref{StrongGood} and \lemref{CliqueSize} together imply:

\begin{thm}
\thmlabel{StrongLayeredTreewidth}
Every graph obtained by clique-sums of  strongly $(g,p,k,a)$-almost-embeddable graphs is $a+(k+1)(2g+2p+3)$-good, has layered treewidth at most $a+(k+1)(2g+2p+3)$, and admits layered separations of width $a+(k+1)(2g+2p+3)$. 
\end{thm}

\lemref{TreewdithDiameter} and \thmref{StrongLayeredTreewidth}  together imply:

\begin{thm}
\thmlabel{kAlmost}
Let $G$ be a graph obtained by clique-sums of strongly $k$-almost-embeddable graphs. Then:

\vspace*{-3ex}
\begin{enumerate}[(a)]
\item $G$ is $(4k^2+8k+3)$-good, 
\item $G$ has layered treewidth at most $4k^2+8k+3$, 
\item $G$ admits layered separations of width $4k^2+8k+3$, and 
\item if $G$ has diameter $d$ then $G$ has  treewidth less than $(4k^2+8k+3)(d+1)$. 
\end{enumerate}
\end{thm}

\thmref{kAlmost}(d)  improves upon a result by \citet[Proposition~10]{Grohe-Comb03} who proved an upper bound on the treewidth of $d\cdot f(k)$, where $f(k)\approx k^k$. Moreover, this result of \citet{Grohe-Comb03} assumes there are no apex vertices. That is, it is for clique-sums of $(k,k,k,0)$-almost-embeddable graphs. 

Recall that a graph $H$ is \emph{apex} if $H-v$ is planar for some vertex $v$ of $H$. \citet{DvoTho} proved  a structure theorem for general $H$-minor-free graphs, which in the case of apex graphs $H$, says that $H$-minor-free graphs are obtained from clique-sums of strongly $k$-almost-embeddable graphs, for some $k=k(H)$; see \citep{DHK} for related claims. Thus \thmref{kAlmost} implies:

\begin{thm} 
\thmlabel{ApexLayering}
For each fixed apex graph $H$ there is a constant $\ell=\ell(H)$ such that every $H$-minor-free graph has layered treewidth at most $\ell$ and admits layered separations of width $\ell$.
\end{thm}

We now characterise the minor-closed classes  with bounded layered treewidth.

\begin{thm}
\thmlabel{Equivalent}
The following are equivalent for a proper minor-closed class of graphs $\mathcal{G}$:
\begin{enumerate}[(1)]
\item every graph in $\mathcal{G}$ has bounded layered treewidth,
\item every graph in $\mathcal{G}$ admits layered separations of bounded width,
\item $\mathcal{G}$ has linear local treewidth,
\item $\mathcal{G}$ has bounded local treewidth,
\item $\mathcal{G}$ excludes a fixed apex graph as a minor, 
\item there exists $k\in\N$ such that every graph in $\mathcal{G}$ is obtained from clique-sums of strongly $k$-almost-embeddable graphs.
\end{enumerate}
\end{thm}

\begin{proof}
\lemref{DecompLayering} shows that (1) implies (2). 
\lemref{LayerSepLinearLocal} shows that (2) implies (3), which implies (4) by definition. 
\citet{Eppstein-Algo00}  proved that (4) and (5) are equivalent; see \citep{DH-Algo04} for an alternative proof. 
As mentioned above, \citet{DvoTho} proved  that (5) implies (6). 
\thmref{kAlmost}(b) proves that (6) implies (1). 
\end{proof}

Note that \citet{DH-SODA04} previously proved that (3) and (4) are equivalent. Also note that the minor-closed assumption in \thmref{Equivalent} is essential: \citet{DEW16} proved that the $n\times n \times n$ grid has bounded local treewidth but has unbounded, indeed $\Omega(n)$, layered treewidth. 

\section{Rich  Decompositions and Shadow-Complete Layerings}
\seclabel{GeneralMinor}

As observed in \secref{Vortices}, it is not the case that graphs in every proper minor-closed class admit layered separations of bounded width. However, in this section we introduce some tools (namely, rich tree decompositions and shadow-complete layerings) that enable our methods based on layered tree decompositions to be  extended to conclude results about graphs excluding a fixed minor or fixed topological minor. See \twothmref{GeneralTrack}{TopoMinorNonRep} for two applications of the results in this section.

A tree decomposition $(B_x\subseteq V(G):x\in V(T))$ of a graph $G$ is \emph{$k$-rich} if $B_x\cap B_y$ is a clique in $G$ on at most $k$ vertices, for each edge $xy\in E(T)$. Rich tree decomposition are implicit in the graph minor structure theorem, as demonstrated by the following lemma. 

\begin{lem}
\lemlabel{ProduceRichDecomp}
For every fixed graph $H$ there are constants $k\geq 1$ and $\ell\geq 1$ depending only on $H$, such that every $H$-minor-free graph $G_0$ is a spanning subgraph of a graph $G$ that has a $k$-rich tree decomposition such that each bag induces an $\ell$-almost-embeddable subgraph of $G$. 
\end{lem}

\begin{proof}
By \thmref{GMST}, there is a constant $\ell=\ell(H)$ such that $G_0$ has a tree decomposition $\mathcal{T}:=(B_x\subseteq V(G):x\in V(T))$ in which each torso is $\ell$-almost-embeddable. Let $G$ be the graph obtained from $G$ by adding a clique on  $B_x\cap B_y$ for each edge $xy\in E(T)$. Let $\mathcal{T'}$ be the tree decomposition of $G$  obtained from $\mathcal{T}$. Each bag of $\mathcal{T'}$ is the torso of the corresponding bag of $\mathcal{T}$, and thus induces an $\ell$-almost-embeddable subgraph of $G$. By \lemref{CliqueSize}, there is a constant $k$ depending only on $\ell$ such that every clique in an $\ell$-almost embeddable graph has size at most $k$. Thus $\mathcal{T'}$ is a $k$-rich tree decomposition of $G$. 
\end{proof}

Consider a layering $(V_0,V_1,\dots,V_t)$ of a graph $G$. Let $H$ be a connected component of $G[V_i\cup V_{i+1}\cup \dots\cup V_t]$, for some $i\in\{1,\dots,t\}$. The \emph{shadow} of  $H$ is the set of vertices in $V_{i-1}$ adjacent to  $H$. The layering is \emph{shadow-complete} if every shadow is a clique. This concept was introduced by \citet{KP-DM08} and implicitly by \citet{DMW05}. It is a key to the proof that graphs of bounded treewidth have bounded nonrepetitive chromatic number  \citep{KP-DM08} and bounded  track-number \citep{DMW05}. 

The following lemma generalises a result by \citet{KP-DM08}, who proved it when each bag of the tree decomposition is a clique (that is, for chordal graphs). We allow bags to induce more general graphs, and in subsequent sections we apply this lemma with each bag inducing an $\ell$-almost-embeddable graph (\twothmref{GeneralTrack}{TopoMinorNonRep}).

For a subgraph $H$ of a graph $G$, a tree decomposition  $(C_y\subseteq V(H):y\in V(F))$  of  $H$ is  \emph{contained in} a tree decomposition  $(B_x\subseteq V(G):x\in V(T))$ of $G$ if for each bag $C_y$ there is a bag $B_x$ such that $C_y\subseteq B_x$. 

\begin{lem}
\lemlabel{RichShadow}
Let $G$ be a graph with a  $k$-rich tree decomposition $\mathcal{T}$ for some $k\geq 1$. Then $G$ has a shadow-complete layering $(V_0,V_1,\dots,V_t)$ such that every shadow has size at most $k$, and for each $i\in\{0,\dots,t\}$, the subgraph $G[V_i]$ has a  $(k-1)$-rich tree decomposition contained in $\mathcal{T}$.
\end{lem}

\preproof\begin{proof} 
We may assume that $G$ is connected with at least one edge. Say  $\mathcal{T}=(B_x\subseteq V(G):x\in V(T))$  is a $k$-rich tree decomposition of $G$. If $B_x\subseteq B_y$ for some edge $xy\in E(T)$, then contracting $xy$ into $y$ (and keeping bag $B_y$) gives a new $k$-rich tree decomposition of $G$. Moreover, if a tree decomposition of a subgraph of $G$ is contained in the new tree decomposition of $G$, then it is contained in the original. Thus we may assume that  $B_x\not\subseteq B_y$ and $B_y\not\subseteq B_x$ for each edge $xy\in V(T)$.

Let $G'$ be the graph obtained from $G$ by adding an edge between every pair of vertices in a common bag (if the edge does not already exist). Let $r$ be a vertex of $G$. Let $\alpha$ be a node of $T$ such that $r\in B_\alpha$. Root $T$ at $\alpha$. Now every non-root node of $T$ has a parent node. Since $G$ is connected, $G'$ is connected. For $i\geq 0$, let $V_i$ be the  set of vertices of $G$ at distance $i$ from $r$ in $G'$. Thus, for some $t$,  $(V_0,V_1,\dots,V_t)$ is a layering of $G'$ and also of $G$ (since $G\subseteq G'$). 

Since each bag $B_x$ is a clique in $G'$,  $V_1$ is the set of vertices of $G$ in bags that contain $r$ (not including $r$ itself). More generally, $V_i$ is the set of vertices $v$ of $G$ in bags that intersect $V_{i-1}$ such that $v$ is not in $V_0\cup\dots\cup V_{i-1}$.

Define $B'_\alpha:=B_\alpha\setminus\{r\}$ and $B''_\alpha:=\{r\}$. For a non-root node $x\in V(T)$ with parent node $y$, define $B'_x:=B_x\setminus B_y$ and $B''_x:=B_x\cap B_y$. Since $B_x\not\subseteq B_y$, it follows that $B'_x\neq\emptyset$.  One should think that $B'_x$ is the set of vertices that first appear in $B_x$ when traversing down the tree decomposition from the root, while $B''_x$ is the set of vertices in $B_x$ that appear above $x$ in the tree decomposition. 

Consider a node $x$ of $T$. Since $B_x$ is a clique in $G'$, $B_x$ is contained in at most two consecutive layers. Consider (not necessarily distinct) vertices $u,v$ in the set $B'_x$, which is not empty. Then the distance between $u$ and $r$ in $G'$ equals the distance between $v$ and $r$ in $G'$. Thus $B'_x$ is contained in one layer, say $V_{\ell(x)}$. Let $w$ be the neighbour of $v$ in some shortest path between $v$ and $r$ in $G'$. Then $w$ is in $B''_x\cap V_{\ell(x)-1}$. In conclusion, each bag $B_x$ is contained in precisely two consecutive layers, $V_{\ell(x)-1}\cup V_{\ell(x)}$, such that $\emptyset\neq B'_x\subseteq V_{\ell(x)}$ and $B_x\cap V_{\ell(x)-1}\subseteq B''_x\neq\emptyset$. Also, observe that if $y$ is an ancestor of $x$ in $T$, then $\ell(y)\leq\ell(x)$. Call this property $(\star)$. 

We now prove that $G[V_i]$ has the desired $(k-1)$-rich tree decomposition. Since $G[V_0]$ has one vertex and no edges, this is trivial for $i=0$. Now assume that $i\in\{1,\dots,t\}$.

Let $T_i$ be the subgraph of $T$ induced by the nodes $x$ such that $\ell(x)\leq i$. By property $(\star)$, $T_i$ is a (connected) subtree of $T$. We claim that $\mathcal{T}_i:=(B_x\cap V_i:x\in V(T_i))$ is a $T_i$-decomposition of $G[V_i]$. First we prove that each vertex $v\in V_i$ is in some bag of $\mathcal{T}_i$. Let $x$ be the node of $T$ closest to $\alpha$ such that $v\in B_x$. Then $v\in B'_x$ and $\ell(x)=i$. Hence $v$ is in the bag $B_x\cap V_i$ of $\mathcal{T}_i$, as desired. 

Now  we prove that for each edge $vw\in E(G[V_i])$, both $v$ and $w$ are in a common bag of $\mathcal{T}_i$. Let $x$ be the node of $T$ closest to $\alpha$ such that $v\in B_x$. Let $y$ be the node of $T$ closest to $\alpha$ such that $w\in B_y$. Thus $v\in B'_x$ and $x\in V(T_i)$, and $w\in B'_y$ and $y\in V(T_i)$. Since $vw\in E(G)$, there is a bag $B_z$ containing both $v$ and $w$, and $z$ is a descendant of both $x$ and $y$ in $T$ (by the definition of $x$ and $y$). Without loss of generality, $x$ is on the $y\alpha$-path in $T$. Moreover, $v$ is also in $B_y$ (since $v$ and $w$ are in a common bag of $\mathcal{T}$). Thus $v$ and $w$ are in the bag  $B_y\cap V_i$ of $\mathcal{T}_i$, as desired. 

Finally, we prove that for each vertex $v\in V_i$, the set of bags in $\mathcal{T}_i$ that contain $v$ correspond to a (connected) subtree of $T_i$. By assumption, this property holds in $T$. Let $X$ be the subtree of $T$ whose corresponding bags in $\mathcal{T}$ contain $v$. Let $x$ be the root of $X$. Then $v\in B'_x$ and $\ell(x)=i$. By property $(\star)$, $\ell(z)\geq i$ for each node $z$ in $X$. Moreover, again by property $(\star)$, deleting from $X$ the nodes $z$ such that $\ell(z)\geq i+1$ gives a connected subtree of $X$, which is precisely the subtree of $T_i$ whose bags in $\mathcal{T}_i$ contain $v$. 

Hence $\mathcal{T}_i$ is a $T_i$-decomposition of $G[V_i]$. By definition, $\mathcal{T}_i$ is contained in $\mathcal{T}$. 

We now prove that $\mathcal{T}_i$ is $(k-1)$-rich. Consider an edge $xy\in E(T_i)$. Without loss of generality, $y$ is the parent of $x$ in $T_i$. Our goal is to prove that $B_x \cap B_y\cap V_i=B''_x\cap V_i$ is a clique on at most $k-1$ vertices. Certainly, it is a clique on at most $k$ vertices, since $\mathcal{T}$ is $k$-rich. Now, $\ell(x)\leq i$ (since $x\in V(T_i)$). If $\ell(x)<i$ then $B_x\cap V_i=\emptyset$, and we are done. Now assume that $\ell(x)=i$. Thus $B'_x\subseteq V_i$ and $B'_x\neq\emptyset$. Let $v$ be a vertex in $B'_x$. Let $w$ be the neighbour of $v$ on a shortest path in $G'$ between $v$ and $r$. Thus $w$ is in $B''_x\cap V_{i-1}$. Thus  $|B''_x\cap V_i|\leq k-1$, as desired. Hence $\mathcal{T}_i$ is $(k-1)$-rich. 

We now prove that $(V_0,V_1,\dots,V_t)$ is shadow-complete. Let $H$ be a connected component of $G[V_i\cup V_{i+1}\cup\dots\cup V_t]$ for some $i\in\{1,\dots,t\}$. Let $X$ be the subgraph of $T$ whose corresponding bags in $\mathcal{T}$ intersect $V(H)$. Since $H$ is connected, $X$ is indeed a connected subtree of $T$. Let $x$ be the root of $X$. 
Consider a vertex $w$ in the shadow of $H$. That is, $w\in V_{i-1}$ and $w$ is adjacent to some vertex $v$ in $V(H)\cap V_i$. 
Let $y$ be the node closest to $x$ in $X$ such that $v\in B_y$. Then $v\in B'_y$ and $w\in B''_y$. Thus $\ell(y)=i$. 
Note that $B_x\subseteq V_{\ell(x)-1} \cup V_{\ell(x)}$ and some vertex in $B_x$ is in $V(H)$ and is thus in $V_i\cup V_{i+1}\cup\dots \cup V_t$. Thus $\ell(x)\geq i$. Since $x$ is an ancestor of $y$ in $T$, 
$\ell(x)\leq \ell(y)=i$ by property $(\star)$, implying $\ell(x)=i$. Thus $w\in B''_x$. Since $B''_x$ is a clique, the shadow of $H$ is a clique. Hence $(V_0,V_1,\dots,V_t)$ is shadow-complete. 
Moreover, since $|B''_x|\leq k$, the shadow of $H$ has size at most $k$. 
\end{proof}

\section{Track and Queue Layouts}
\seclabel{TrackQueue}

The results of this section are expressed in terms of track layouts of graphs, which is a type of graph layout closely related to queue layouts and 3-dimensional grid drawings. A \emph{vertex} $|I|$\emph{-colouring} of a graph $G$ is a partition $\{V_i:i\in I\}$ of $V(G)$ such that for every edge $vw\in E(G)$, if $v\in V_i$ and $w\in V_j$ then $i\ne j$. The elements of the set $I$ are \emph{colours}, and  each set $V_i$ is a \emph{colour class}. Suppose that  $\preceq_i$ is a total order on
each colour class $V_i$.  Then each pair $(V_i,\preceq_i)$ is a \emph{track},  and $\{(V_i,\preceq_i):i\in I\}$ is an $|I|$-\emph{track assignment} of $G$. 
                               
An \emph{X-crossing} in a track assignment consists of two edges $vw$ and $xy$
such that $v\prec_ix$ and $y\prec_jw$, for distinct colours $i$ and $j$.  A
$t$-track assignment of $G$ that has no X-crossings is called a
$t$-\emph{track layout} of $G$. The minimum $t$ such that a graph $G$
has $t$-track layout is called the \emph{track-number} of $G$, denoted by $\tn(G)$. 
\citet{DMW05} proved that  
\begin{equation}
\label{QueueTrack}
\qn(G)\leq\tn(G)-1\enspace.
\end{equation}
Conversely, \citet{DPW04} proved that $\tn(G)\leq f(\qn(G))$ for some function $f$. In this sense, queue-number and track-number are tied. 

As described in \secref{IntroQueueDrawing}, \citet{Duj15} recently showed that layered separators can be used to construct queue layouts. In fact, the construction produces a track layout, which with \eqref{QueueTrack} gives the desired bound for queue layouts. 

\begin{lem}[\citep{Duj15}]
\lemlabel{LayeredSeparatorTrackQueue}
If a graph $G$ admits layered separations of width $\ell$ then 
$$\qn(G)<\tn(G)\leq 3\ell(\ceil{\log_{3/2}n}+1)\enspace.$$
\end{lem}

Recall the following result discussed in \secref{tool}.

\begin{lem}[\citep{DFJW13,LT79}]
\lemlabel{LayeredSeparatorPlanar}
Every planar graph admits layered separations of width $2$.  
\end{lem}

\twolemref{LayeredSeparatorTrackQueue}{LayeredSeparatorPlanar} imply the following result of \citet{Duj15}.

\begin{thm}[\citep{Duj15}]
Every $n$-vertex planar graph $G$ satisfies $$\qn(G)<\tn(G)\leq 6\ceil{\log_{3/2}n}+6\enspace.$$
\end{thm}


Now consider queue and track layouts of graphs with Euler genus $g$. \thmref{GenusLayering} and \lemref{LayeredSeparatorTrackQueue} imply that  $\qn(G)<\tn(G)\in \Oh{g \log n}$. This bound can be improved to $\Oh{g+\log n}$  as follows. A straightforward extension of the proof of \lemref{LayeredSeparatorTrackQueue} gives the following result; see \appref{New} for a proof.

\begin{lem}
\lemlabel{LayeredTreewidthTrackQueue}
Let $\mathcal{T}$ be a tree decomposition of a graph $G$ such that there is a set $Q\subseteq V(G)$ with at most $\ell_1$ vertices in each layer of some layering of $G$, and $\mathcal{T}$ restricted to $G-Q$ has layered width at most $\ell_2$ with respect to the same layering. Then
$$\qn(G)<\tn(G)\leq 3\ell_1 + 3\ell_2(1+\log_{3/2}n)\enspace.$$
\end{lem}

\thmref{NewGenusDecomposition} and \lemref{LayeredTreewidthTrackQueue} with $\ell_1=2g$ and $\ell_2=3$ imply the following generalisation of the above results.

\begin{thm}
\thmlabel{GenusTrack}
For every $n$-vertex graph $G$ with Euler genus $g$,  $$\qn(G)<\tn(G)\leq 6g + 9(1+\log_{3/2}n)\enspace.$$
\end{thm}

\thmref{ApexLayering} and \lemref{LayeredSeparatorTrackQueue} imply the following further generalisation.

\begin{thm}
\thmlabel{ApexTrack}
For each fixed apex graph $H$, for every $n$-vertex $H$-minor-free graph $G$,  $$\qn(G)<\tn(G)\leq \Oh{\log n}\enspace.$$
\end{thm}

We now extend this result to arbitrary proper minor-closed classes. \citet{DMW05} implicitly proved that if a graph $G$ has a shadow-complete layering such that each layer induces a subgraph with track-number at most $c$ and each shadow has size at most $s$, then $G$ has track-number at most $3c^{s+1}$; see \appref{Implicit}.  Iterating this result gives the next lemma.

\begin{lem}[implicit in \citep{DMW05}]
\lemlabel{ShadowTrack}
For some number $c$, let $\mathcal{G}_0$ be a class of graphs with track-number at most $c$. For $k\geq1$, let $\mathcal{G}_k$ be a class of graphs that have a shadow-complete layering such that each shadow has size at most $k$, and each layer induces a graph in $\mathcal{G}_{k-1}$. Then every graph in $\mathcal{G}_k$ has track-number at most $3^{(k+1)!-1}c^{(k+1)!}$. 
\end{lem}

\begin{lem}
\lemlabel{RichTrack}
Let $G$ be a graph that has a $k$-rich tree decomposition $\mathcal{T}$ such that the subgraph induced by each bag has a $c$-track layout. Then $G$ has a $3^{(k+1)!-1}c^{(k+1)!}$-track layout.
\end{lem}

\begin{proof}
For $j\in\{0,\dots,k\}$, let $\mathcal{G}_j$ be the set of induced subgraphs of $G$ that have a $j$-rich tree decomposition contained in $\mathcal{T}$. Note that $G$ itself is in $\mathcal{G}_k$.  Consider a graph $G'\in \mathcal{G}_0$. Then $G'$ is the union of disjoint subgraphs of $G$, each of which is contained in a bag of $\mathcal{T}$ and thus has a $c$-track layout. Thus  $G'$ has a $c$-track layout. Consider some $G'\in \mathcal{G}_j$ for some $j\in\{1,\dots,k\}$. Thus $G'$ is an induced subgraph of $G$ with a $j$-rich tree decomposition contained in $\mathcal{T}$. By \lemref{RichShadow}, $G'$ has a shadow-complete layering $(V_0,\dots,V_t)$ such that for each layer $V_i$, the induced subgraph $G'[V_i]$  has a $(j-1)$-rich tree decomposition $\mathcal{T}_i$ contained in $\mathcal{T}$. Thus $G'[V_i]$ is in $\mathcal{G}_{j-1}$. By \lemref{ShadowTrack}, the graph $G$ has a 
$3^{(k+1)!-1}c^{(k+1)!}$-track layout. 
\end{proof}

\begin{thm}
\thmlabel{GeneralTrack}
For every fixed graph $H$, every $H$-minor-free $n$-vertex graph has track-number and queue-number at most  $\log^{\Oh{1}}n$.
\end{thm}

\preproof\begin{proof}  
Let $G_0$ be an $H$-minor-free graph on $n$ vertices. By \lemref{ProduceRichDecomp}, there are constants $k\geq 1$ and $\ell\geq 1$ depending only on $H$, such that $G_0$ is a spanning subgraph of a graph $G$ that has a $k$-rich tree decomposition $\mathcal{T}$ such that each bag induces an $\ell$-almost-embeddable subgraph of $G$. To layout one such $\ell$-almost-embeddable subgraph, put each of the at most $\ell$ apex vertices on its own track, and layout the remaining graph with $3(4\ell^2+8\ell+3)(\ceil{\log_{3/2}n}+1)$ tracks by \thmref{kAlmost} and \lemref{LayeredSeparatorTrackQueue}. (Here we do not use the clique-sums or apices in \thmref{kAlmost}.)\ By \lemref{RichTrack} with $c=\ell+3(4\ell^2+8\ell+3)(\ceil{\log_{3/2}n}+1)$, our graph $G$ and thus $G_0$ has track-number at most  $3^{(k+1)!-1}(\ell+3(4\ell^2+8\ell+3)(\ceil{\log_{3/2}n}+1))^{(k+1)!}$,  which is in $\log^{\Oh{1}}n$ since $k$ and $\ell$ are constants (depending only on $H$). The claimed bound on queue-number follows from \eqref{QueueTrack}.
\end{proof}

\section{3-Dimensional Graph Drawing}
\seclabel{3DimDrawing}

This section presents our results for 3-dimensional graph drawings, which are based on the following connection between track layouts and 3-dimensional graph drawings.

\begin{lem}[\citep{DMW05,DujWoo-SubQuad-AMS}]
\lemlabel{Track2Drawing}
If a $c$-colourable $n$-vertex graph $G$ has a $t$-track layout, then $G$ has 3-dimensional grid drawings with \Oh{t^2n} volume and  with \Oh{c^7tn} volume.
\end{lem}

Every graph with Euler genus $g$ is \Oh{\sqrt{g}}-colourable \citep{Heawood1890}. Thus 
\thmref{GenusTrack} and \lemref{Track2Drawing} imply:

\begin{thm}
\thmlabel{GenusDrawing}
Every $n$-vertex graph with Euler genus $g$ has a 3-dimensional  grid drawing with volume \Oh{g^{7/2}(g+ \log n)n}.
\end{thm}

For fixed $H$, every $H$-minor-free graph is \Oh{1}-colourable \citep{Mader67}. Thus \thmref{ApexTrack} and \lemref{Track2Drawing} imply:

\begin{thm}
\thmlabel{ApexDrawing}
For each fixed apex graph $H$, every $n$-vertex $H$-minor-free graph has a 3-dimensional  grid drawing with volume \Oh{n \log n}.
\end{thm}

\lemref{Track2Drawing} and  \thmref{GeneralTrack} extend this theorem to arbitrary proper minor-closed classes:

\begin{thm}
\thmlabel{MinorDrawing}
For each fixed graph $H$, every $H$-minor-free $n$-vertex graph has a 3-dimensional grid drawing with volume $n\log^{\Oh{1}}n$.
\end{thm}

The best previous upper bound on  the volume of 3-dimensional  grid drawings of graphs with bounded Euler genus or   $H$-minor-free graphs was \Oh{n^{3/2}} \citep{DujWoo-SubQuad-AMS}.

\section{Nonrepetitive Colourings}
\seclabel{NonRep}

This section proves our results for nonrepetitive colourings. Recall the following two results by \citet{DFJW13} discussed in \secref{NonrepetitiveColourings}. (\thmref{NonRepPlanar} is implied by   \twolemref{LayeredSeparatorPlanar}{NonRep}.)\ 

\begin{lem}[\cite{DFJW13}]
  \lemlabel{NonRep}
 If an $n$-vertex graph $G$   admits layered separations of width $\ell$ then $$\pi(G)\leq 4\ell(1+\log_{3/2}n)\enspace.$$
\end{lem}

\begin{thm}[\citep{DFJW13}]
\thmlabel{NonRepPlanar}
  For every $n$-vertex planar graph $G$,$$\pi(G)\leq  8(1+\log_{3/2}n)\enspace.$$
\end{thm}

Now consider nonrepetitive colourings  of graphs $G$ with Euler genus $g$. \thmref{GenusLayering}  and \lemref{NonRep} imply that $\pi(G)\leq \Oh{g\log n}$. This bound can be improved to $\Oh{g+\log n}$  as follows. A straightforward extension of the proof of \lemref{NonRep} gives the following result; see \appref{New} for a proof. 

\begin{lem}
\lemlabel{LayeredTreewidthNonRep}
Let $\mathcal{T}$ be a tree decomposition of a graph $G$ such that there is a set $Q\subseteq V(G)$ with at most $\ell_1$ vertices in each layer of some layering of $G$, and $\mathcal{T}$ restricted to $G-Q$ has layered width at most $\ell_2$ with respect to the same layering. Then
$$\pi(G)\leq 4\ell_1 + 4\ell_2(1+\log_{3/2}n)\enspace.$$
\end{lem}

\thmref{NewGenusDecomposition} and \lemref{LayeredTreewidthNonRep} with $\ell_1=2g$ and $\ell_2=3$ imply the following generalisation of the above results.

\begin{thm}
\thmlabel{GenusNonRep}
For every $n$-vertex graph with Euler genus $g$,  $$\pi(G)\leq 8g + 12(1+\log_{3/2}n)\enspace.$$
\end{thm}

To generalise \thmref{GenusNonRep}, we employ a result by \citet{KP-DM08}. They proved that if a graph $G$ has a shadow-complete layering such that the graph induced by each layer is nonrepetitively $c$-colourable, then $G$ is nonrepetitively $4c$-colourable \citep[Theorem~6]{KP-DM08}. Iterating this result gives the next lemma.

\begin{lem}[\citep{KP-DM08}] 
\lemlabel{ShadowNonRep}
For some number $c$, let $\mathcal{G}_0$ be a class of graphs with nonrepetitive chromatic number at most $c$. For $k\geq1$, let $\mathcal{G}_k$ be a class of graphs that have a shadow-complete layering such that each layer induces a graph in $\mathcal{G}_{k-1}$. Then every graph in $\mathcal{G}_k$ has nonrepetitive chromatic number at most $c\,4^k$.
\end{lem}

\twolemref{RichShadow}{ShadowNonRep} lead to the following result:

\begin{lem}
\lemlabel{RichColour}
Let $G$ be a graph that has a $k$-rich tree decomposition $\mathcal{T}$ such that the subgraph induced by each bag is nonrepetitively $c$-colourable. Then $G$ is $c\,4^k$-colourable.
\end{lem}

\preproof\begin{proof}
For $j\in\{0,\dots,k\}$, let $\mathcal{G}_j$ be the set of induced subgraphs of $G$ that have a $j$-rich tree decomposition contained in $\mathcal{T}$. Note that $G$ itself is in $\mathcal{G}_k$. Consider a graph $G'\in \mathcal{G}_0$. Then $G'$ is the union of disjoint subgraphs of $G$, each of which is contained in a bag of $\mathcal{T}$ and is thus nonrepetitively $c$-colourable. Thus  $G'$ is nonrepetitively $c$-colourable. Now consider some $G'\in \mathcal{G}_j$ for some $j\in\{1,\dots,k\}$. Thus $G'$ is an induced subgraph of $G$ with a $j$-rich tree decomposition contained in $\mathcal{T}$. By \lemref{RichShadow}, $G'$ has a shadow-complete layering $(V_0,\dots,V_t)$ such that for each layer $V_i$, the induced subgraph $G'[V_i]$  has a $(j-1)$-rich tree decomposition $\mathcal{T}_i$ contained in $\mathcal{T}$. Thus $G'[V_i]$ is in $\mathcal{G}_{j-1}$. By \lemref{ShadowNonRep}, the graph $G$ is nonrepetitively $4^kc$-colourable.
\end{proof}


\lemref{RichColour} can be used to prove that  every $n$-vertex graph excluding a fixed minor is nonrepetitively \Oh{\log n}-colourable. The proof is  analogous to that of \thmref{GeneralTrack} for track layouts. 
However, in the setting of nonrepetitive colourings, we obtain a stronger result for graphs excluding a fixed topological minor. 
The following two results are the key tools. The first is a structure theorem for excluded topological minors due to \citet{GM15}.

\begin{thm}[\citep{GM15}] 
\thmlabel{GroheMarx}
For every graph $H$ there is a constant $k$ such that every graph excluding $H$ as a topological minor has a tree decomposition such that each torso  is $k$-almost-embeddable or has at most $k$ vertices with degree greater than $k$. 
\end{thm}

\citet{AGHR-RSA02} proved that graphs with maximum degree $\Delta$ are nonrepetitively $\Oh{\Delta^2}$-colourable. The best known bound is due to \citet{DJKW16}.

\begin{thm}[\cite{DJKW16}]
\thmlabel{NonRepColourDegree}
Every graph with maximum degree $\Delta\geq 2$ is nonrepetitively $\pi(\Delta)$-colourable, 
where
$$\pi(\Delta)\leq \CEIL{\left(1+\frac{1}{\Delta^{1/3}-1}+\frac{1}{\Delta^{1/3}}\right)\Delta^2} \leq \Delta^2+4\Delta^{5/3}.$$
\end{thm}

%
%


\begin{thm}
\thmlabel{TopoMinorNonRep}
For every fixed graph $H$, every $H$-topological-minor-free $n$-vertex graph is nonrepetitively \Oh{\log n}-colourable.
\end{thm}

\preproof\begin{proof} 
Let $G_0$ be an $H$-topological-minor-free graph on $n$ vertices. It follows from \thmref{GroheMarx} that there are constants $k\geq 1$ and $\ell\geq 1$ depending only on $H$, such that $G_0$ is a spanning subgraph of a graph $G$ that has a $k$-rich tree decomposition $\mathcal{T}$ such that the subgraph induced by each bag is $\ell$-almost-embeddable or has at most $\ell$ vertices with degree greater than $\ell$. (The proof is analogous to that of \lemref{ProduceRichDecomp}, using the fact that a graph with at most $\ell$ vertices of degree greater than $\ell$ contains no $K_{\ell+2}$ subgraph.)\ Define $c:=\ell+4(4\ell^2+8\ell+3)(1+\log_{3/2}n)$. Let $G'$ be the subgraph induced by some bag of $\mathcal{T}$. Then $G'$ is is $\ell$-almost-embeddable or has at most $\ell$ vertices of degree greater than $\ell$. If $G'$ is $\ell$-almost-embeddable, then give each of the at most $\ell$ apex vertices its own colour and colour the remainder with $c-\ell$ colours by \thmref{kAlmost} and \lemref{NonRep}.  (Here we do not use the clique-sums or apices in \thmref{kAlmost}.)\ Otherwise, $G'$ has at most $\ell$ vertices of degree greater than $\ell$, in which case give each of the at most $\ell$ vertices with degree greater than $\ell$ its own colour and colour the remainder with $\ell^2+4\ell^{5/3}$ colours by \thmref{NonRepColourDegree}. Note that  $\ell^2+4\ell^{5/3}+\ell \leq c$. Thus  $G'$ is nonrepetitively $c$-colourable. By \lemref{RichShadow}, the graph $G$ is nonrepetitively $4^kc$-colourable, as is $G_0$, since $G_0$ is a subgraph of $G$.
\end{proof}

Note that if $H$ has maximum degree at least 4, then a $\log^{\Oh{1}} n$ bound for graphs excluding $H$ as a topological minor is not possible for track-number or queue-number. In this case, every graph with maximum degree 3 does not contain $H$ as a topological minor. But \citet{Wood-QueueDegree} proved that for $\Delta\geq 3$ and sufficiently large $n$ there exists $n$-vertex graphs with maximum degree $\Delta$ and with track-number and queue-number at least $c\sqrt{\Delta}n^{1/2-1/\Delta}$, for some constant $c$. In particular there are cubic graphs with track-number and queue-number at least $cn^{1/6}$.

\section{Reflections}
\seclabel{Reflections}

1.~We now show that the statement of \thmref{ApexLayering} implies the Grid Minor Theorem of \citet{RS-GraphMinorsV-JCTB86}, which says that for every planar graph $H$ there is an integer $c$ such that every $H$-minor-free $G$ graph has treewidth at most $c$. Let $H^+$ be the apex graph obtained from $H$ by adding a dominant vertex $v$. Let $G^+$ be the graph obtained from $G$ by adding a dominant vertex $x$. Suppose that $G^+$ contains an $H^+$-minor. We may assume that $x$ is the image of some vertex $w$ of $H^+$ in the $H^+$-minor, implying $G$ contains $H^+-w$ as a minor. Note that $H^+-w$ contains a subgraph isomorphic to $H$ (since $v$ is dominant in $H^+$). Thus $G$ contains $H$ as a minor, which is a contradiction. Hence $G^+$ is $H^+$-minor-free. By \thmref{ApexLayering}, $G^+$ has layered treewidth at most some $\ell=\ell(H)$. Since $G^+$ has radius 1, at most three layers are used. Thus $G^+$ and $G$ have treewidth less than $3\ell$, and the Grid Minor Theorem holds. In this light, \thmref{ApexLayering} can be viewed as a qualitative strengthening of the Grid Minor Theorem. On the other hand, since the proof of \thmref{ApexLayering} depends on the Graph Minor Structure Theorem, which in turn depends on the Grid Minor Theorem, it is desirable to find a proof of \thmref{ApexLayering} that does not depend on the Graph Minor Structure Theorem and gives reasonable bounds on the layered treewidth.


2.~Local treewidth has been successfully applied in the fields of approximation algorithms and bidimensionality \citep{Baker94,DH-SJDM04,Grohe-Comb03,DH-SODA04}. Given that layered tree decompositions can be thought of as a global structure for graphs of bounded local treewidth, it would be interesting to see if layered treewidth has algorithmic applications. See \citep{DvorakThin} for results in this direction. 

3.~While this paper has focused on the layered treewidth of minor-closed graph classes, various non-minor-closed graph classes also have bounded layered treewidth. For example, in a follow-up paper, \citet{DEW16} proved that graphs that can be drawn on a surface with Euler genus $g$ with at most $k$ crossings per edge have layered treewidth at most $(4g+6)(k+1)$. Similar results are obtained for map graphs. 

4.~The similarity between queue/track layouts and nonrepetitive colourings is remarkable given how different the definitions seem at first glance. Both parameters have bounded expansion \citep{NOW} and admit very similar properties with respect to subdivisions \citep{NOW,DujWoo-DMTCS05}. Many proof techniques work for both queue/track layouts and nonrepetitive colourings, in particular layered separations and shadow-complete layerings. One exception is that graphs of bounded maximum degree have bounded nonrepetitive chromatic number \citep{AGHR-RSA02,Gryczuk-IJMMS07,HJ-DM11,DJKW16}, whereas  graphs of bounded maximum degree have unbounded track- and queue-number \citep{Wood-QueueDegree}. It would be interesting to prove a more direct relationship. Do graphs of bounded track/queue-number have bounded nonrepetitive chromatic number? More specifically, do 1-queue graphs  have bounded nonrepetitive chromatic number? And do 3-track graphs  have bounded nonrepetitive chromatic number? 

5.~Finally, we mention the work of \citet{Shahrokhi13} who introduced a definition equivalent to layered treewidth. (We became aware of reference \citep{Shahrokhi13}  when it was posted on the arXiv in 2015.)\ \citet{Shahrokhi13}  was motivated by questions completely different from those in the present paper. In our language, he proved that for every graph $G$ with layered treewidth $k$, there is a graph $G_1$ with clique cut width at most $2k-1$ and a chordal graph $G_2$ such that $G=G_1\cap G_2$. \citet{Shahrokhi13} then proved that every planar graph $G$ has layered treewidth at most 4, implying that there is a graph $G_1$ with clique cut width at most $7$ and a chordal graph $G_2$ such that $G=G_1\cap G_2$. \thmref{GenusDecomposition} with $g=0$ improves these bounds from 4 to 3 and thus from 7 to 5. All our other results about layered treewidth can be applied in this domain as well. 

\subsection*{Acknowledgements} This research was partially completed at Bellairs Research Institute in Barbados. Thanks to Zden{\v{e}}k Dvo{\v{r}}{\'a}k, Gwena\"el Joret, Sergey Norin, Bruce Reed and Paul Seymour for helpful discussions. Thanks to the anonymous referees for numerous helpful comments.


\def\soft#1{\leavevmode\setbox0=\hbox{h}\dimen7=\ht0\advance \dimen7
  by-1ex\relax\if t#1\relax\rlap{\raise.6\dimen7
  \hbox{\kern.3ex\char'47}}#1\relax\else\if T#1\relax
  \rlap{\raise.5\dimen7\hbox{\kern1.3ex\char'47}}#1\relax \else\if
  d#1\relax\rlap{\raise.5\dimen7\hbox{\kern.9ex \char'47}}#1\relax\else\if
  D#1\relax\rlap{\raise.5\dimen7 \hbox{\kern1.4ex\char'47}}#1\relax\else\if
  l#1\relax \rlap{\raise.5\dimen7\hbox{\kern.4ex\char'47}}#1\relax \else\if
  L#1\relax\rlap{\raise.5\dimen7\hbox{\kern.7ex
  \char'47}}#1\relax\else\message{accent \string\soft \space #1 not
  defined!}#1\relax\fi\fi\fi\fi\fi\fi}

\appendix

\section{Recursive Separators}
\applabel{New}

Here we prove \twolemref{LayeredTreewidthTrackQueue}{LayeredTreewidthNonRep}. The method, which is based on recursive application of layered separations,  is a straightforward generalisation of the method of \citet{DFJW13} for nonrepetitive colouring and of \citet{Duj15} for track layouts. Both lemmas have the same starting assumptions: Let $V_1,V_2,\dots,V_p$ be a layering of a graph $G$.  Let $\mathcal{T}$ be a tree decomposition  of  $G$ such that there is a set $Q\subseteq V(G)$ with at most $\ell_1$ vertices in each layer $V_i$, and $\mathcal{T}$ restricted to $G-Q$ has layered width at most $\ell_2$ with respect to $V_1,V_2,\dots,V_p$ . 

For each vertex $v\in Q$, let $\depth(v):=0$. For $i\in\{1,\dots,p\}$, injectively label the vertices in $V_i\cap Q$ by $1,2,\dots,\ell_1$. Let $\lab(v)$ be the label assigned to each vertex $v\in V_i \cap Q$. By assumption, $G-Q$ has layered treewidth at most $\ell_2$ and thus admits layered separations of width $\ell_2$ by \lemref{DecompLayering}. Now run the following recursive algorithm \textsc{Compute}$(V(G)\setminus Q,1)$. 

\begin{center}
  \framebox{
\begin{minipage}{\textwidth-8mm}
    \medskip
      \textsc{Compute} (input $S$ and $d$, where $S \subseteq V(G)\setminus Q$ and $d \in \mathbb{Z}^+$)

\hspace*{-1.5em}
\begin{minipage}{\textwidth+1mm}
\medskip
      \begin{enumerate}
      \item If $S=\emptyset$ then exit.
      \item Let $(G_1,G_2)$ be a separation of $G-Q$ such that each 
      layer $V_i$ contains at most $\ell_2$ vertices in $V(G_1 \cap G_2)\cap S$, and both $V(G_1) \setminus V(G_2)$ and $V(G_2) \setminus V(G_1)$
        contain at most $\frac{2}{3}|S|$ vertices in $S$.
      \item Let $\depth(v):=d$ for each vertex $v\in V(G_1 \cap G_2)\cap S$.
      \item For $i\in\{1,\dots,p\}$, injectively label the vertices in
        $V_i\cap V(G_1 \cap G_2)\cap S$ by $1,2,\dots,\ell_2$. Let
        $\lab(v)$ be the label assigned to each vertex $v\in
        V_i\cap V(G_1 \cap G_2)\cap S$.
      \item\textsc{Compute}$((V(G_1) \setminus V(G_2))\cap S,d+1)$
      \item \textsc{Compute}$((V(G_2) \setminus V(G_1))\cap S,d+1)$ \smallskip
      \end{enumerate}

    \end{minipage}
        \end{minipage}}
\end{center}

\bigskip The recursive application of \textsc{Compute} determines a rooted
binary tree $T$, where each node of $T$ corresponds to one call to
\textsc{Compute}. Associate each vertex whose depth and label is
computed in a particular call to \textsc{Compute} with the
corresponding node of $T$. (Observe that the depth and label of each
vertex is determined exactly once.)\ Note that the maximum depth is at most
$1+\log_{3/2}n$. 

\begin{proof}[Proof of \lemref{LayeredTreewidthTrackQueue}] 
Our goal is to prove that $\tn(G)\leq 3\ell_1 + 3\ell_2(1+\log_{3/2}n)$. 
The tracks are indexed by triples of integers as follows. 
Colour each vertex $v$ by $(\blah(v),\depth(v),\lab(v))$, 
where $\blah(v):=i\bmod{3}$ if $v\in V_i$, 
and $\depth$ and $\lab$ are computed above. 
This defines a track assignment for $G$. 
We now order each track. 
Consider two vertices $v\in V_i$ and $w\in V_j$  on the same track; that is, 
$(\blah(v),\depth(v),\lab(v))=(\blah(w),\depth(w),\lab(w))$.
If $i<j$ then place $v\prec w$ in the track. 
If $j<i$ then place $w\prec v$ in the track. 
Now assume that $i=j$. 
If $v$ and $w$ are associated with the same node of $T$, 
then $i = j$ implies $\lab(v) \neq \lab(w)$, which is a contradiction.
Now assume $v$ and $w$ are associated with distinct nodes of $T$ with least common ancestor $\alpha$. 
Say $S$ was the input set corresponding to $\alpha$, and  $(G_1,G_2)$ was the corresponding separation of $G-Q$. 
Without loss of generality, $v\in (V(G_1) \setminus V(G_2))\cap S$ and  $w\in (V(G_2) \setminus V(G_1))\cap S$.
Place $v\prec w$ in the track. It is easily seen that each track is totally ordered by $\preceq$. 

Suppose on the contrary that $(\blah(v),\depth(v),\lab(v))=(\blah(w),\depth(w),\lab(w))$ for some edge $vw$ of $G$. Say $v\in V_i$ and $w\in V_j$. Thus $i\equiv j\pmod{3}$ and $|i-j|\leq 1$, implying $i=j$. Since $\depth(v)=\depth(w)$ and $vw\in E(G)$, it must be that $v$ and $w$ are associated with the same node of $T$, implying  $\lab(v) \neq \lab(w)$, which is a contradiction. Thus the track assignment is a proper colouring. 
 
We now show there is no X-crossing. 
Suppose that edges $vw$ and $xy$ form an X-crossing, where
$(\blah(v),\depth(v),\lab(v))=(\blah(x),\depth(x),\lab(x))$ and
$(\blah(w),\depth(w),\lab(w))=(\blah(y),\depth(y),\lab(y))$ and
$v\prec x$ and $y\prec w$. 
Say $v\in V_a$ and $w\in V_b$ and $x\in V_c$ and $y\in V_d$. Since $vw$ and $xy$ are edges, $|a-b|\leq 1$ and $|c-d|\leq 1$. 
Since $\blah(v)=\blah(x)$ and  $\blah(w)=\blah(y)$ we have $a\equiv c\pmod{3}$ and $b\equiv d\pmod{3}$. 
Since $v\prec x$ and $y\prec w$ we have $a\leq c$ and $d\leq b$. 
If $a<c$ then $a+3\leq c\leq d+1\leq b+1\leq a+2$, which is a contradiction. 
Similarly, if $d<b$ then $d+3\leq b\leq a+1\leq c+1\leq d+2$, which is a contradiction. Now assume that $a=c$ and $d=b$. 
Without loss of generality, $\depth(v)=\depth(x)\leq \depth(w)=\depth(y)$. 
Since $\lab(v)=\lab(x)$ and $v\neq x$, it follows that $v$ and $x$ are associated with distinct nodes of $T$.
Let $\alpha$ be the least common ancestor of these nodes of $T$.
Say $S$ was the input set corresponding to $\alpha$, and  $(G_1,G_2)$ was the corresponding separation of $G-Q$. 
Since $v\prec x$ we have $v\in (V(G_1) \setminus V(G_2))\cap S$ and  $x\in (V(G_2) \setminus V(G_1))\cap S$.
Since $\depth(v)\leq \depth(w)$ and $vw$ is an edge,  $w\in (V(G_1) \setminus V(G_2))\cap S$.
Similarly,  since $\depth(x)\leq \depth(y)$ and $xy$ is an edge,  $y\in (V(G_2) \setminus V(G_1))\cap S$. 
Therefore the algorithm places $w\prec y$ on their track, which is a contradiction. 
Hence no two edges form an X-crossing.
The number of tracks is at most $3\ell_1 + 3\ell_2(1+\log_{3/2}n)$. 
\end{proof}

\begin{proof}[Proof of \lemref{LayeredTreewidthNonRep}] 
Our goal is to prove that  $\pi(G)\leq 4\ell_1 + 4\ell_2(1+\log_{3/2}n)$. 
\citet{KP-DM08}  proved that for every layering of a  graph $G$, there is a (not necessarily proper) 4-colouring of $G$
  such  that   for every repetitively coloured path $(v_1,v_2,\dots,v_{2t})$, the
  subpaths $(v_1,v_2,\dots,v_{t})$ and
  $(v_{t+1},v_{t+2},\dots,v_{2t})$ have the same layer pattern (that is, for $i\in\{1,\dots,t\}$, 
  vertices $v_i$ and $v_{t+i}$ are in the same layer). Let $\blah$ be a such a 4-colouring. 
  Now colour each vertex $v$ by $(\blah(v),\depth(v),\lab(v))$, where $\depth$ and $\lab$ are computed above. 
Suppose on
the contrary that $(v_1,v_2,\dots,v_{2t})$ is a repetitively coloured
path in $G$. Then $(v_1,v_2,\dots,v_{t})$ and
$(v_{t+1},v_{t+2},\dots,v_{2t})$ have the same layer pattern.  In
addition, $\depth(v_i)=\depth(v_{t+i})$ and
$\lab(v_{i})=\lab(v_{t+i})$ for all $i\in[1,t]$.
Let $v_i$ and $v_{t+i}$ be vertices in this path with minimum depth.
Since $v_i$ and $v_{t+i}$ are in the same layer and have the same
label, these two vertices were not labelled at the same step of the
algorithm.  Let $x$ and $y$ be the two nodes of $T$ respectively
associated with $v_i$ and $v_{t+i}$. Let $z$ be the least common
ancestor of $x$ and $y$ in $T$. Say node $z$ corresponds to call
\textsc{Compute}$(B,d)$. Thus $v_i$ and $v_{t+i}$ are in $B$ (since if
a vertex $v$ is in $B$ in the call to $\textsc{Compute}$ associated
with some node $q$ of $T$, then $v$ is in $B$ in the call to
$\textsc{Compute}$ associated with each ancestor of $q$ in $T$).  Let
$(G_1,G_2)$ be the separation in \textsc{Compute}$(B,d)$. Since
$\depth(v_i)=\depth(v_{t+i})>d$, neither $v_i$ nor $v_{t+i}$ are in
$V(G_1 \cap G_2)$. Since $z$ is the least common ancestor of $x$ and
$y$, without loss of generality, $v_i\in V(G_1) \setminus V(G_2)$ and
$v_{t+i}\in V(G_2) \setminus V(G_1)$.
Thus some vertex $v_j$ in the subpath
$(v_{i+1},v_{i+2},\dots,v_{t+i-1})$ is in $V(G_1 \cap G_2)$. If
$v_j\in B$ then $\depth(v_j)=d$. If $v_j\not\in B$ then
$\depth(v_j)<d$. In both cases,
$\depth(v_j)<\depth(v_i)=\depth(v_{t+i})$, which contradicts the
choice of $v_i$ and $v_{t+i}$. Hence there is no repetitively coloured
path in $G$. There are $4\ell_1$ colours at depth 0 and $4\ell_2$ colours at every other depth. 
Since the maximum depth is at most
$1+\log_{3/2}n$, the number of colours is at most
$4\ell_1 + 4\ell_2(1+\log_{3/2}n)$. 
\end{proof}

Note that in both \twolemref{LayeredTreewidthTrackQueue}{LayeredTreewidthNonRep} we may replace $\log_{3/2}n$ by $\log_2n$ by using separators (and the first part of \lemref{RS}) instead of separations (as in the second part of \lemref{RS}).

\section{Track Layout Construction}
\applabel{Implicit}

Here we sketch a proof of a result used in \secref{TrackQueue} that is implicit in the work of  \citet{DMW05}.

\begin{lem}[implicit in \citep{DMW05}]
If a graph $G$ has a shadow-complete layering $V_1,\dots,V_t$ such that each layer induces a subgraph with track-number at most $c$ and each shadow has size at most $s$, then $G$ has track-number at most $3c^{s+1}$.  
\end{lem}

\begin{proof}[Proof Sketch]
Let $T$ be the graph obtained from $G$ by contracting each connected component of each subgraph $G[V_i]$ into a single node. For each node $x$ of $T$, let $H_x$ be the corresponding connected component. Let $V_i'$ be the vertices of $T$ arising from $V_i$. Thus $V_1',\dots,V_t'$ is a layering of $T$. For each node $y\in V'_i$ where $i\in\{1,\dots,t\}$, let $C_y$ be the set of neighbours of $H_y$ in $V_{i-1}$. We may assume that $C_y\neq\emptyset$. Since the given layering is shadow-complete, $C_y$ is a clique, called the \emph{parent clique} of $y$. Now $C_y$ is contained in a single connected component $H_x$ of $G[V_{i-1}]$, for some node $x\in V'_{i-1}$. Call $x$ the \emph{parent node} and $H_x$ the \emph{parent component} of $y$. This shows that each node in $V'_i$ has exactly one neighbour in $V'_{i-1}$, which implies that $T$ is a forest. As illustrated in \figref{TreeTrackLayout}, $T$ has  a 3-track layout $T_0,T_1,T_2$.

\begin{figure}[!ht]
\begin{center}
\includegraphics{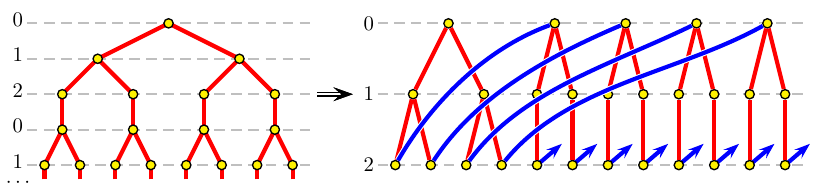}
\caption{A $3$-track layout of $T$.}
\figlabel{TreeTrackLayout}
\end{center}
\end{figure}

By assumption, for each node $x$ of $T$, there is a $c$-track layout of $H_x$. For a clique $C$ of $H_x$ of size at most $s$, define the \emph{signature} of $C$ to be the set of (at most $s$) tracks that contain $C$. Since there is no X-crossing, the set of cliques of $H_x$ with the same signature can be linearly ordered as $C_1\prec \dots\prec C_p$ so that if $v$ and $w$ are vertices in the same track and in distinct cliques $C_i$ and $C_j$ with $i<j$, then $v\prec w$ in that track. Call this a \emph{clique ordering}. 

Replace each track $T_j$ of $T$ by $c$ sub-tracks, and replace each node $x\in T_j$ by the $c$-track layout of $H_x$. This defines a $3c$ track assignment for $G$. Clearly an edge in some $H_x$ crosses no other edge. Two edges between a parent component $H_x$ and the same child component $H_y$ do not form an X-crossing, since the endpoints in $H_x$ of such edges form a clique (the parent clique of $y$), and therefore are in distinct tracks. The only possible X-crossing is between edges $ab$ and $cd$, where $a$ and $c$ are in some parent component $H_x$, and $b$ and $d$ are in distinct child components $H_y$ and $H_z$, respectively. 

To solve this problem, when determining the 3-track layout of $T$, the child nodes of each node $x$ are ordered in their track so that $y\prec z$ whenever the parent cliques $C_y$ and $C_z$ have the same signature, and $C_y\prec C_z$ in the clique ordering. Then group the child nodes of $x$ according to the signatures of their parent cliques, and for each signature $\sigma$, use a distinct set of $c$ tracks for the child components whose parent cliques have signature $\sigma$. Now the ordering of the child components with the same signature agrees with the clique ordering of their parent cliques, and therefore agrees with the ordering of any neighbours in the parent component. It follows that there is no X-crossing. The number of tracks is at most $3c$ times the number of signatures, which is at most $\sum_{i=1}^s\binom{c}{i}\leq c^s$.  In total there are at most $3c\cdot c^s$ tracks. 
\end{proof}

This proof makes no effort to reduce the number of tracks. Various tricks due to \citet{DMW05} and \citet{DLMW-DM09} make a modest improvement. 

\end{document}